\documentclass[12pt]{article}

\usepackage{amssymb}
\usepackage{amscd}
\usepackage{amsmath}
\usepackage{amsmath}
\usepackage[all,cmtip,arrow]{xy}
\usepackage{amsmath}
\usepackage{amssymb}
\usepackage{amsthm}
\usepackage{changes}

\usepackage{amsfonts}
\usepackage{amscd}
\usepackage[OT2,T1]{fontenc}
\DeclareSymbolFont{cyrletters}{OT2}{wncyr}{m}{n}
\DeclareMathSymbol{\Sha}{\mathalpha}{cyrletters}{"58}

\newtheorem{thm}{Theorem}[section]
\newtheorem{prop}[thm]{Proposition}
\newtheorem{cor}[thm]{Corollary}
\newtheorem{corollary}[thm]{Corollary}
\newtheorem{lem}[thm]{Lemma}

\newtheorem{remark}[thm]{Remark}
\newtheorem{remarks}[thm]{Remarks}

\theoremstyle{definition}
\newtheorem{defn}[thm]{Definition}

\newcommand{\IZ}{\text{${\mathbb{Z}}$}}
\newcommand{\IQ}{\text{${\mathbb{Q}}$}}

\DeclareMathOperator{\cok}{cok}

\DeclareMathOperator{\Ext}{Ext}

\DeclareMathOperator{\Fr}{Fr}
\DeclareMathOperator{\Gal}{Gal}

\DeclareMathOperator{\ind}{ind}

\DeclareMathOperator{\Hom}{Hom}

\DeclareMathOperator{\rk}{rk}
\DeclareMathOperator{\Sel}{Sel}

\DeclareMathOperator{\Tr}{Tr}

\newcommand{\CC}{\mathbb{C}}

\newcommand{\RR}{\mathbb{R}}
\newcommand{\ZZ}{\mathbb{Z}}

\newcommand{\bc}{\mathbb{C}}
\newcommand{\bq}{\mathbb{Q}}
\newcommand{\br}{\mathbb{R}}
\newcommand{\bz}{\mathbb{Z}}

\newcounter{condone}

\newcommand{\mpar}[1]{}

\newcommand{\lra}{\longrightarrow}

\newcommand{\Zp}{{\mathbb{Z}_p}}
\newcommand{\Qp}{{\mathbb{Q}_p}}
\newcommand{\ZpG}{\mathbb{Z}_p[G]}
\newcommand{\ZG}{\mathbb{Z}[G]}
\newcommand{\CpG}{\mathbb{C}_p[G]}
\newcommand{\QpG}{\mathbb{Q}_p[G]}
\newcommand{\Cp}{{\mathbb{C}_p}}
\newcommand{\Ze}{{\mathbb{Z}}}
\newcommand{\Ce}{{\mathbb{C}}}
\renewcommand{\Re}{{\mathbb{R}}}

\newcommand{\PJj}{P_{(J,j)}}
\newcommand{\PtJj}{P^t_{(J,j)}}
\newcommand{\Ptj}{P_{(t,j)}}
\newcommand{\Pttj}{P^t_{(t,j)}}

\newcommand{\Puk}{P_{(u,k)}}

\newcommand{\Qu}{\mathbb{Q}}
\newcommand{\QG}{\Qu[G]}
\newcommand{\calP}{\mathcal{P}}
\newcommand{\NT}{\mathrm{NT}}

\newcommand{\calL}{\mathcal{L}}
\newcommand{\calM}{\mathcal{M}}

\newcommand{\frp}{\mathfrak{p}}

\newcommand{\etncp}{{eTNC$_p$} }

\newcommand{\tensorZ}{\otimes_\Ze}

\newcommand{\tensorZp}{\otimes_\Zp}

\newcommand{\directlim}{{\lim\limits_{\longrightarrow}}}

\newcommand{\TpFA}{{T_{p, F}(A)}}

\newcommand{\Cf}{{ C_{A,F}^{f, \bullet} }}

\newcommand{\sseq}{\subseteq}
\newcommand{\trivchar}{{{\bf 1}_G}}
\newcommand{\nr}{{\mathrm{nr}}}
\newcommand{\EG}{{\mathbb{E}[G]}}
\newcommand{\RG}{{\mathbb{R}[G]}}

\renewcommand{\mod}{\mathrm{mod}\ }

\title{Congruences for critical values of
higher derivatives of twisted Hasse-Weil $L$-functions}
\author{Werner Bley and Daniel Macias Castillo}
\begin{document}

\maketitle

\begin{abstract}Let $A$ be an abelian variety over a number field $k$ and $F$ a finite cyclic extension of $k$ of $p$-power degree for an odd prime $p$.
Under certain technical hypotheses, we obtain a reinterpretation of the equivariant Tamagawa number conjecture (`eTNC') for $A$, $F/k$ and $p$ as an explicit family of
$p$-adic congruences involving values of derivatives of the Hasse-Weil $L$-functions of twists of $A$, normalised by completely explicit twisted regulators. This
reinterpretation makes the eTNC amenable to numerical verification and furthermore leads to explicit predictions which refine well-known conjectures of Mazur and Tate.
\end{abstract}

\section{Introduction}\label{intro}

Let $A$ be an abelian variety of dimension $d$ defined over a number field $k$. We write $A^t$ for the dual abelian variety.
Let $F/k$ be a finite Galois extension with group $G := \Gal(F/k)$. We let $A_F$ denote the base change of $A$ through $F/k$ and consider the motive
$M_F := h^1(A_F)(1)$ as a motive over $k$ with a natural action of the semi-simple $\Qu$-algebra $\QG$.

We will study the equivariant Tamagawa number conjecture as formulated by Burns and Flach in \cite{bufl01} for the pair
$(M_F, \ZG)$. This conjecture asserts the validity of an equality in the relative algebraic $K$-group $K_0(\ZG, \Re[G])$. If $p$
is a prime, we refer to the image of this equality in $K_0(\ZpG, \CpG)$ as the `eTNC$_p$ for $(M_F, \ZG)$' (here $\Cp$ denotes the completion of an algebraic closure of
$\Qp$). If $p$ does not
divide the order of $G$ then the ring $\ZpG$ is regular and one can use the techniques described in \cite[\S 1.7]{bufl96}
to give an explicit interpretation of this projection. In this manuscript we will focus on primes $p$ dividing the order of $G$, for which such an
interpretation is in general very difficult to obtain.

In \cite{bmw}, a close analysis of the finite
support cohomology of Bloch and Kato for the base change of the $p$-adic Tate module of $A^t$ is carried out
under certain technical hypotheses on $A$ and $F$.
A consequence of this analysis is an explicit reinterpretation of the \etncp in terms of a natural `equivariant regulator'
(see \cite[Th.~5.1]{bmw}). The main results of the present manuscript are based on the computation of this
equivariant regulator in the special case where $F/k$ is cyclic of degree $p^n$ for an odd prime $p$. Under certain
additional hypotheses on the structure of Tate-Shafarevich groups of $A$ over the intermediate fields of $F/k$
we obtain a completely explicit interpretation of the eTNC$_p$ (see Theorem \ref{ZpG theorem}). Whilst this is of independent theoretical
interest, it also makes the \etncp amenable to numerical verifications.

One of the main motivations behind our study of the equivariant Tamagawa number conjecture for the pair
$(M_F, \ZG)$ is the hope that this conjecture may provide a coherent overview of and a systematic approach to the study of properties of leading terms
and values at $s=1$ of Hasse-Weil $L$-functions. In order to describe our current steps in this direction, we first recall the general philosophy of
`refined conjectures of the Birch and Swinnerton-Dyer type' that originates in the work of Mazur and Tate in \cite{mazurtate}. These conjectures concern,
for elliptic curves $A$ defined over $\bq$ and certain abelian groups $G$, the properties of `modular elements' $\theta_{A,G}$ belonging a priori to the
rational group ring $\bq[G]$ and constructed from the modular symbols associated to $A$, therefore interpolating the values at $s=1$ of the twisted
Hasse-Weil $L$-functions associated to $A$ and $G$. More precisely, the aim is to predict the precise
power $r$ (possibly infinite) of the augmentation ideal $I$ of the integral group ring $\bz[G]$ with the property that $\theta_{A,G}$ belongs to $I^r$
but not to $I^{r+1}$, and furthermore to describe the image of $\theta_{A,G}$ in the quotient $I^r/I^{r+1}$ (whenever such an integer $r$ exists).
In the process of studying the modular element $\theta_{A,G}$, Mazur and Tate also predict that it should belong to the Fitting ideal over $\bz[G]$ of
their `integral Selmer group' $S(A/F)$ (and refer to such a statement as a `weak main conjecture') and ask for a `strong main conjecture'
predicting a generator of the Fitting ideal of an explicitly described natural modification of $S(A/F)$
(see \cite[Remark after Conj. 3]{mazurtate}).


However, it is well-known that in many cases of interest the modular element
$\theta_{A,G}$ vanishes, thus rendering any such properties trivial, and it would therefore be desirable to carry out an analogous study for
elements interpolating leading terms rather than values at $s=1$ of the relevant Hasse-Weil $L$-functions, normalised by appropriate regulators.
Although the aim to study such elements already underlies the results of \cite{bmw}, one of the main advantages of confining ourselves to the special case in which
the given extension of number fields $F/k$ is cyclic of prime-power degree is that we are led to defining completely explicit `twisted regulators' from our computation
of the aforementioned equivariant regulator of \cite{bmw}. Furthermore, we arrive at very 
explicit statements without having to restrict ourselves to situations in which
the relevant Mordell-Weil groups are projective when considered as Galois modules. In particular, we derive predictions of the following nature for such an element
$\calL$ that interpolates leading terms at $s=1$ of twisted Hasse-Weil $L$-functions normalised by our twisted regulators from the assumed validity of the 
eTNC$_p$ for $(M_F, \ZG)$:\begin{itemize}\item a formula for the precise power
$h\in\bz_{\geq 0}$ of the augmentation ideal $I_{G,p}$ of the integral group ring $\bz_p[G]$ with the property that $\calL$ belongs to $I_{G,p}^h$ but not
to $I_{G,p}^{h+1}$ (expressed in terms of the ranks of the Mordell-Weil groups of $A$ over the intermediate fields of $F/k$),
and a formula for the image of $\calL$ in the quotient $I_{G,p}^h/I_{G,p}^{h+1}$ (see Corollary \ref{Cor 1});
\item the statement that the element
$\calL$ of $\bz_p[G]$ (resp. a straightforward modification of $\calL$) annihilates the $p$-primary Tate-Shafarevich group of $A^t$ (resp. $A$) over $F$
as a Galois module (see Theorem \ref{Ann theorem} and Corollary \ref{Cor 2}); \item  and the explicit description of a natural quotient of (the Pontryagin dual
of) the $p$-primary Selmer group of $A$ over $F$ whose Fitting ideal is generated by $\calL$ (see Theorem \ref{Ann theorem}).\end{itemize}

The structure of the paper is as follows. In Section \ref{statements} we present our main results and in Section \ref{Proofs} we supply
the proofs.  In order to prepare for the proofs  we recall in Section \ref{bmw reformulation} the relevant material from \cite{bmw}.
In the final Section \ref{examples} we present some numerical computations.

We would like to thank David Burns and Christian Wuthrich for some helpful discussions concerning this project, and the referee for making several useful suggestions.

\subsection{Notations and setting}

We mostly adapt the notations from \cite{bmw}.


For a finite group $\Gamma$ we write $D(\Zp[\Gamma])$ for the derived category of complexes of left $\Zp[\Gamma]$-modules.
We also write $D^p(\Zp[\Gamma])$ for the full triangulated subcategory of $D(\Zp[\Gamma])$ comprising complexes that are
perfect (that is, isomorphic in $D(\Zp[\Gamma])$ to a bounded complex of finitely generated projective $\Zp[\Gamma]$-modules).

We also write $\hat{\Gamma}$ for the set of irreducible $E$-valued characters of $\Gamma$, where
$E$ denotes either $\Ce$ or $\Cp$ (we will throughout our arguments have fixed an isomorphism of fields $j:\Ce\to\Cp$ and use it to implicitly identify both sets,
with the intended meaning of $\hat{\Gamma}$ always clear from the context). We let
${\bf{1}}_\Gamma$ denote the trivial character of $\Gamma$ and write $\check\psi$ for the contragrediant character of each $\psi \in \hat{\Gamma}$.
We write
\[
e_\psi = \frac{\psi(1)}{|\Gamma|} \sum_{\gamma \in \Gamma} \psi(\gamma) \gamma^{-1}
\]
for the idempotent associated with $\psi \in \hat{\Gamma}$ and also set ${\rm Tr}_\Gamma:=\sum_{\gamma \in \Gamma} \gamma$.

For any abelian group $M$ we let $M_{{\rm tor}}$ denote its torsion subgroup and $M_{{\rm tf}}$ the torsion-free quotient  $M / M_{{\rm tor}}$.
We also set $M_p := \Zp \tensorZ M$ and, if $M$ is finitely generated, we
set $\rk(M) := \dim_\Qu(\Qu \tensorZ M)$.

For any $\Zp[\Gamma]$-module $M$ we write $M^\vee$ for the Pontryagin dual $\Hom_\Zp(M, \Qp/\Zp)$ and $M^*$ for the linear dual
$\Hom_\Zp(M, \Zp)$, each endowed with the natural contragredient action of $\Gamma$. Explicitly, for a homomorphism $f$ and elements $m \in M$ and $\gamma \in \Gamma$,
one has $(\gamma f)(m) = f(\gamma^{-1} m)$.

For any Galois extension of fields $L/K$ we abbreviate $\Gal(L/K)$ to $G_{L/K}$. We fix an algebraic closure $K^c$ of $K$ and
abbreviate $G_{K^c/K}$ to $G_K$. For each non-archimedian place $v$ of a number field we write $\kappa_v$ for the residue field.

Throughout
this paper, we will consider the following situation. We have fixed an odd prime $p$ and a Galois extension $F/k$ of number fields with
group $G = G_{F/k}$. Except in Section \ref{bmw reformulation}, the extension $F/k$ will always be cyclic of degree $p^n$. We give ourselves an abelian variety $A$
of dimension $d$ defined over $k$. For each intermediate field $L$ of $F/k$ we write $S_p^L, S_r^L$ and $S_b^L$ for the set of non-archimedean
places of $L$ which are $p$-adic, which ramify in $F/L$ and at which $A/L$ has bad reduction respectively. Similarly, we write
$S_\infty^L,S_\br^L$ and $S_\bc^L$ for the sets of archimedean, real and complex places of $L$ respectively. If $L=k$ we simply write
$S_p, S_r,S_b,S_\infty,S_\br$ and $S_\bc$.

Finally, we write $A(L)$ for the Mordell-Weil group and $\Sha_p(A_L)$ for the $p$-primary Tate-Shafarevich group of $A$ over $L$.

\mpar{statements}
\section{Statement of the main results}\label{statements}

Recall that $A$ is an abelian variety of dimension $d$ defined over the number field $k$. Furthermore, $F/k$ is cyclic of
degree $p^n$ where $p$ is an odd prime.

We assume throughout this section that $A/k$ and $F/k$ are such that

\begin{itemize}
\item [(a)] $p \nmid |A(k)_{{\rm tor}}| \cdot  |A^t(k)_{{\rm tor}}|$,
\item [(b)]  $p\nmid\prod_{v\in S_{{\rm b}}}c_v(A,k)$, where $c_v(A,k)$ denotes the Tamagawa number of $A$ at $v$,
\item [(c)] $A$ has good reduction at all $p$-adic places of $k$,
\item [(d)] $p$ is unramified in $F / \Qu$,
\item [(e)] No place of bad reduction for $A$ is ramified in $F/k$, i.e. $S_b \cap S_r = \emptyset$,
\item [(f)] $p\nmid\prod_{v\in S_{{\rm r}}}|A(\kappa_v)|$,
\item [(g)] The Tate-Shafarevich group $\Sha(A_{F})$ is finite,
\item [(h)] $\Sha_p(A_{F^H}) = 0$ for all non-trivial subgroups $H$ of $G$.
\end{itemize}

\begin{remarks}
Our assumptions (a) - (g) recover the hypotheses (a) - (h) of \cite{bmw}. For a fixed abelian variety $A/k$, the hypotheses (a), (b)
and (c) clearly exclude only finitely many choices of odd prime $p$, while the additional hypotheses (d), (e) and (f) constitute a mild restriction on the choice of cyclic
field extension $F$ of $k$ of odd, prime-power degree. In order to further illustrate this point, we let $S$ denote any finite set of places of $k$ at which $A$ has good
reduction. We then define a set $\Sigma(S)$ of rational primes as the union of the set of all prime divisors $\ell$ of
\[
|A(k)_{{\rm tor}}| \cdot  |A^t(k)_{{\rm tor}}|\cdot \prod_{v\in S_{{\rm b}}}c_v(A,k)\cdot \prod_{v\in S}|A(\kappa_v)|
\] 
and the set of all primes $\ell$ with the property
that $A$ has bad reduction at an $\ell$-adic place of $k$. The set
$\Sigma(S)$ is then clearly finite and, for any odd prime $p \not\in \Sigma(S)$ and any cyclic field extension $F$ of $k$ of $p$-power degree which is unramified
outside $S$ and with the property that $p$ is unramified in $F/\Qu$, all of the hypotheses (a)-(f) are satisfied.

The hypothesis (g) is famously
conjectured to be true in all cases, and it is straightforward to produce specific 
examples for which all of the other hypotheses, including the additional hypothesis (h), are
satisfied (see also Section \ref{examples} and Remark \ref{shaorpermutation} in this regard). We emphasize that in (h) we allow $\Sha_p(A_F)$ to be non-trivial.
\end{remarks}

An understanding of the $G$-module structure of the relevant Mordell-Weil groups is key to our approach. We hence begin by applying a result of Yakovlev \cite{yakovlev}
in order to obtain such explicit descriptions. This approach is inspired by work of Burns, who first obtained a similar result in \cite[Prop. 7.2.6(i)]{leading}.
For a non-negative integer $m$ and a $\ZpG$-module $M$ we write $M^{<m>}$ for the direct sum of $m$
copies of $M$. Furthermore, we set $[m] := \{ 1, \ldots, m \}$.

\begin{prop}\label{MW iso}\mpar{MW iso}
 There exist isomorphisms of $\IZ_p[G]$-modules of the form
\[
A(F)_p \cong\bigoplus_{J\leq G}\IZ_p[G/J]^{<m_J>}\cong A^t(F)_p,
\]
for a set of non-negative integers $\{m_J:J\leq G\}$.
\end{prop}


Proposition \ref{MW iso} has the following immediate consequence for the ranks of the relevant Mordell-Weil groups.

\begin{corollary}\label{ranks}
  For any subgroup $H$ of $G$ we have
\begin{eqnarray*}
&& {\rm rk}(A(F^H))={\rm rk}(A^t(F^H))= \nonumber \\
&=& \sum_{J>H}|G/J|m_J+|G/H|\sum_{J\leq H}m_J\leq|G/H|{\rm rk}(A(k)).
\end{eqnarray*}
\end{corollary}

Proposition \ref{MW iso} combines with Roiter's Lemma (see \cite[(31.6)]{curtisr}) to imply the existence of points $\PJj \in A(F)$ and $\PtJj \in A^t(F)$ for
$J \le G$ and $j \in [m_J]$ with the property that

\mpar{global points}
\begin{equation}\label{global points}
  \begin{array}{c}
 A(F)_p = \bigoplus_{J \le G} \bigoplus_{j \in [m_J]} \Zp[G/J] \PJj,\,\,\,\,\,\,\Zp[G/J] \PJj\cong\Zp[G/J],\\
 A^t(F)_p = \bigoplus_{J \le G} \bigoplus_{j \in [m_J]} \Zp[G/J] \PtJj,\,\,\,\,\,\,\Zp[G/J] \PtJj\cong\Zp[G/J].
  \end{array}
\end{equation}
Furthermore, our choice of points as in (\ref{global points}) guarantees that one also has

\mpar{rat global points}
\begin{equation}\label{rat global points}
  \begin{array}{c}
 \Qu \otimes_{\bz} A(F) = \bigoplus_{J \le G} \bigoplus_{j \in [m_J]} \Qu[G/J] \PJj,\,\,\,\,\,\,\Qu[G/J] \PJj\cong\Qu[G/J],\\
 \Qu \otimes_{\bz} A^t(F) = \bigoplus_{J \le G} \bigoplus_{j \in [m_J]} \Qu[G/J] \PtJj,\,\,\,\,\,\,\Qu[G/J] \PtJj\cong\Qu[G/J].
  \end{array}
\end{equation}

We now fix sets
\[
\calP = \{ \PJj\in A(F) : J \le G, j  \in [m_J] \}, \quad
\calP^t = \{ \PtJj\in A^t(F) : J \le G, j  \in [m_J] \},
\]
such that (\ref{rat global points}) holds. For $0 \le t \le n$ we write $H_t$ for the (unique) subgroup of $G$ of order $p^{n-t}$ and
set $\Ptj := P_{(H_t, j)}, \Pttj := P^t_{(H_t, j)}$. We also put $m_t := m_{H_t}$ and $e_{H_t}:=\frac{1}{|H_t|}{\rm Tr}_{H_t}=\frac{1}{|H_t|}\sum_{g\in H_t}g$.
We write $\langle$ , $\rangle_F$ for the N\'{e}ron-Tate height pairing $A(F)\times A^t(F)\to\br$ defined relative to the field $F$ and define a matrix with entries in $\bc[G]$ by setting
\mpar{reg matrix}
\[
R(\calP, \calP^t) := \left( \frac{1}{|H_u|} \sum_{\tau \in G/H_u} \langle \tau\cdot\Puk, \Pttj \rangle_{F} (\tau\cdot e_{H_u}) \right)_{(u,k), (t,j)},
\]
where $(u,k)$ is the row index with $0 \le u \le n$, $k\in[m_u]$, and
$(t,j)$ is the column index with $0 \le t \le n$, $j\in[m_t]$ (we always order sets of the form $\{ (t,j) \colon 0 \le t \le n, j \in [m_t] \}$ lexicographically).
We note that, since each point $\Puk$ belongs to $A(F^{H_u})$, the action of $G/H_u$ on $\Puk$ is well-defined.

For any matrix $A = \left( a_{(u,k), (t,j)}\right)_{(u,k), (t,j)}$ indexed as above we define
\[
A_{t_0} := \left( a_{(u,k), (t,j)}\right)_{(u,k), (t,j), u,t \ge t_0},
\]
with the convention $A_{t_0}=1$ whenever no entries $a_{(u,k), (t,j)}$ with $u,t \ge t_0$ exist.
If $A$ is a matrix with coefficients $a_{ij}$ in $\bc[G]$ or $\bc_p[G]$, then for any $\psi\in\widehat{G}$ we write $\psi(A)$ for the matrix with coefficients
$\psi(a_{ij})$. We also set $R_{t_0}(\calP, \calP^t) = R(\calP, \calP^t)_{t_0}$.

\mpar{minor def}
\begin{defn}
  \label{minor def}
For each character $\psi \in \hat G$
we define $t_\psi \in \{0, \ldots, n\}$ by the equality $\ker(\psi) = H_{t_\psi}$ and call
\mpar{lower psi minor}
\[
  \lambda_\psi(\calP, \calP^t) := \det\left( \psi \left( R_{t_\psi}(\calP, \calP^t) \right) \right).
\]
the 'lower $\psi$-minor' of $R(\calP, \calP^t)$.
\end{defn}

\begin{remark}\label{changeofptsremark}
  It is easy to see that the element $\sum_{\psi \in \widehat G} \lambda_\psi(\calP, \calP^t) e_\psi \in \CC[G]$ depends upon the choice of points
$\calP$ and $\calP^t$ satisfying (\ref{rat global points}) only modulo $\QG^\times$. Similarly, for any given isomorphism of fields $j:\Ce\to\Cp$, it is clear that the element
 $\sum_{\psi \in \widehat G} j(\lambda_\psi(\calP, \calP^t)) e_\psi \in \bc_p[G]$ depends upon the choice
of points $\calP$ and $\calP^t$ satisfying (\ref{global points}) only modulo $\ZpG^\times$.
\end{remark}

For any order $\Lambda$ in $\IQ[G]$ that contains $\IZ[G]$ we let C$(A,\Lambda)$ denote the integrality part
of the equivariant Tamagawa number conjecture (`eTNC' for brevity) for the pair $(h^1(A_F)(1),\Lambda)$ as formulated by Burns and Flach in  \cite[Conj. 4(iv)]{bufl01}.
Similarly, we let  C$(A,\QG)$ denote the rationality part as formulated in \cite[Conj. 4(iii) or Conj. 5]{bufl01}.
We recall that, under the assumed validity of hypothesis $(g)$, C$(A,\Lambda)$ takes the form of an equality in the relative $K$-group $K_0(\Lambda,\RG)$. For each
embedding $j \colon \Re \lra \Cp$ we denote by C$_{p,j}(A, \Lambda)$ the image of this conjectural equality under the
induced map  $K_0(\Lambda,\RG) \lra K_0(\Lambda_p, \CpG)$. We then say that C$_{p}(A, \Lambda)$ is valid if C$_{p,j}(A, \Lambda)$ is valid for every
isomorphism $j:\bc\to\bc_p$.

The eTNC is an equality between analytic and algebraic invariants associated with $A/k$ and $F/k$. In the following we describe and define
the analytic part. We first recall the definition of periods and Galois Gauss sums of \cite[Sec.~4.4]{bmw}.
We fix N\'eron models $\mathcal{A}^t$ for $A^t$ over $\mathcal{O}_k$ and $\mathcal{A}^t_v$ for $A^t_{k_v}$ over $\mathcal{O}_{k_v}$
for each $v$ in $S_p$ and then fix a $k$-basis $\{\omega_{b}\}_{b \in [d]}$ of the space of invariant differentials $H^0(A^t, \Omega^1_{A^t})$ which
gives $\mathcal{O}_{k_v}$-bases of $H^0\bigl(\mathcal{A}^t_v,\Omega_{\mathcal{A}^t_v}^1\bigr)$ for each such $v$ and is also
such that each $\omega_b$ extends to an element of $H^0\bigl(\mathcal{A}^t, \Omega^1_{\mathcal{A}^t}\bigr)$.

For each $v$ in $S_\CC$ we fix a $\ZZ$-basis $\{\gamma_{v,a}\}_{a\in [2d]}$
of $H_1\bigl(\sigma_v(A^t)(\CC),\ZZ\bigr)$. For each $v$ in $S_\RR$ we let $c$ denote complex conjugation and
fix a $\ZZ$-basis $\{\gamma_{v,a}^+\}_{a\in [d]}$ of $H_1\bigl(\sigma_v(A^t)(\CC),\ZZ\bigr)^{c=1}$. For each $v$ in $S_\RR$, resp. $S_\CC$, we then define periods by setting
\[
\Omega_{v}(A/k) := \Biggl\vert\det \biggl( \int_{\gamma_{v,a}^{+}} \omega_b\biggr)_{\!a,b} \Biggr\vert \hbox{, resp. } 
\Omega_v(A/k) := \Biggl\vert\det \biggl( \int_{\gamma_{v,a}} \omega_b , c\Bigl(\int_{\gamma_{v,a}} \omega_b\Bigr) \biggr)_{\!a,b} \Biggr\vert,
\]
where in the first matrix $(a,b)$ runs over $[d]\times [d]$ and in the second matrix $(a,b)$ runs over $[2d]\times [d]$.

In our special case all characters are one-dimensional and, moreover, $|G|$ is odd. Therefore the definitions of \cite {bmw}
simplify and we set
\begin{eqnarray*}
  \Omega(A/k) &:=& \prod_{v \in S_\infty} \Omega_v(A/k), \\
  w_\infty(k) &:=& i^{|S_\Ce|}.
\end{eqnarray*}
For each place $v$ in $S_r$ we write $\bar{I}_v \sseq G$ for the inertia group of $v$ and $\Fr_v$ for the natural Frobenius in $G/\bar{I}_v$.
We define the `non-ramified characteristic' $u_v$ by
\[
u_v(\psi) := \begin{cases} -\psi(\Fr_v^{-1}), & \psi |_{\bar{I}_v} = 1, \\
                           1           , &  \psi |_{\bar{I}_v} \ne 1.
             \end{cases}
\]
and
\[
u(\psi) := \prod_{v \in S_r} u_v(\psi).
\]
For each character $\psi \in \widehat{G}$ we then define the modified Galois-Gauss sum by setting
\[
\tau^*(\Qu,  \ind_k^\Qu(\psi)) := u(\psi) \tau(\Qu, \ind_k^\Qu(\psi)) \in \left( \Qu^c \right)^\times,
\]
where each individual Galois-Gauss sum $\tau(\Qu, \cdot )$ is as defined by Martinet in \cite{martinet}. For each $\psi \in \widehat{G}$
we set
\[
\calL_\psi^* = \calL_{A, F/k, \psi}^* := \frac{L_{S_r}^*(A, \check{\psi}, 1) \tau^*(\Qu,  \ind_k^\Qu(\psi))^d }{\Omega(A/k) w_\infty(k)^d} \in \Ce^\times,
\]
where here for each finite set $\Sigma$ of places of $k$ we write $L^*_\Sigma(A, \psi, 1)$ for the leading term in the Taylor expansion at
$s=1$ of the $\Sigma$-truncated $\psi$-twisted Hasse-Weil-$L$-function of $A$. Without any further mention we will always
assume that the functions $L_\Sigma(A, \psi, s)$ have analytic continuation to $s=1$ (as conjectured in \cite[Conj.~4 (i)]{bufl01}) and recall that they are then expected
to have a zero of order $r_\psi := \dim_\Ce(e_\psi(\Ce \tensorZ A(F)))$ (this is the rank conjecture \cite[Conj.~4 (ii)]{bufl01}).

We finally define
\[
\calL^* = \calL^*_{A, F/k} := \sum_{\psi \in \widehat{G}} \calL^*_{A, F/k, \psi} e_\psi \in \Ce[G]^\times
\]
and note that the element $\calL^*$ defined in \cite[Th.~5.1]{bmw} specialises precisely to our definition.

\mpar{rat theorem}
\begin{thm}\label{rat theorem}
$C(A, \QG)$ is valid if and only if
\[
\mathcal{L}^*_{\psi}\lambda_{\psi}(\mathcal{P},\mathcal{P}^t)^{-1}\in\IQ(\psi)
\]
for all $\psi \in \widehat G$ and furthermore, for any $\gamma\in$ ${\rm Gal}(\IQ(\psi)/\IQ)$,
\[
\mathcal{L}^*_{\psi^{\gamma}}\lambda_{\psi^{\gamma}}(\mathcal{P},\mathcal{P}^t)^{-1}=
\gamma\left(\mathcal{L}^*_{\psi}\lambda_{\psi}(\mathcal{P},\mathcal{P}^t)^{-1}\right),
\] for any, or equivalently every, choice of points $\calP$ and $\calP^t$ such that (\ref{rat global points}) holds.
\end{thm}

\begin{remarks}
(i)  From the definitions of $u(\psi)$, $w_\infty(k)$ and the definition of local Euler factors it is immediately clear that in the statement
of Theorem \ref{rat theorem} we can replace $\calL_\psi^*$ by
\[
\tilde\calL_\psi^* := \frac{L^*(A, \check{\psi}, 1) \tau(\Qu,  \ind_k^\Qu(\psi))^d }{\Omega(A/k)}.
\]

(ii) The explicit conditions on elements of the form $\mathcal{L}^*_{\psi}\lambda_{\psi}(\mathcal{P},\mathcal{P}^t)^{-1}$ given in
Theorem \ref{rat theorem} generalise and refine the predictions given by Fearnley and Kisilevsky in \cite{fka, fk}. For details see
\cite[Ex.~5.2]{bleyone}. In particular, we note that the numerical computations performed by Fearnley and Kisilevsky can be interpreted via
Theorem \ref{rat theorem} as supporting evidence for conjecture $C(A, \QG)$.
\end{remarks}

We fix a generator $\sigma$ of $G$ and define $\Sigma$ to be the diagonal matrix indexed by pairs
$(t,j),(s,i)$ with $\sigma^{p^t} - 1$ at the diagonal entry
associated to $(t,j)$ and zeros elsewhere.
For any matrix $A = \left( a_{(u,k), (t,j)}\right)_{(u,k), (t,j)}$ indexed by tuples $(u,k)$ and $(t,j)$ as above we define
\[
A^{t_0} := \left( a_{(u,k), (t,j)}\right)_{(u,k), (t,j), u,t \le t_0},
\]
once again with the convention $A^{t_0}=1$ whenever no entries $a_{(u,k), (t,j)}$ with $u,t \le t_0$ exist. We recall that for each character $\psi \in \widehat G$
we defined $t_\psi$ such that $\ker(\psi) = H_{t_\psi}$. We define the the 'upper $\psi$-minor' of $\Sigma$ by
\begin{eqnarray*}
  \delta_\psi            &:=& \det\left( \psi \left( \Sigma^{t_\psi-1} \right) \right).
\end{eqnarray*}
It is easy to see that for another choice of generator of $G$, say $\tau$, one has
\[
\sum_{\psi \in \widehat{G}} \frac{\delta_\psi(\sigma)}{\delta_\psi(\tau)} e_\psi \in \ZpG^\times.
\]

Under our current hypotheses on the data $(A,F/k,p)$ and the additional hypothesis that $\Sha_p(A_F) = 0$, and for any intermediate field $L$ of $F/k$, we shall say that
BSD$_p(L)$ holds if, for any choice of $\bz$-bases $\{Q_i\}$ and $\{R_j\}$ of $A(L)$ and $A^t(L)$ respectively and of isomorphism $j:\Ce\to\Cp$, one has that 
\[j\left(\frac{L^*(A/L,1)\cdot(\sqrt{|d_L|})^d}{{\rm det}(\langle Q_i,R_j\rangle_L)\cdot\prod_{v\in S^L_\infty}\Omega_v(A/L)}\right)\in\bz_p^{\times}.
\] Here $d_L$ denotes the discriminant of the field $L$ and each period $\Omega_v(A/L)$
is as defined above but relative to the field $L$ rather than $k$. It will become apparent in the proof of Theorem \ref{max theorem} below that the validity of BSD$_p(L)$ is
equivalent to the validity of the $p$-part of the eTNC for the pair $(h^1(A_L)(1),\bz)$. We recall that hypotheses (a), (b) and (h) justify the fact that no orders of
torsion subgroups of Mordell-Weil groups, Tamagawa numbers or orders of Tate-Shafarevich groups occur in this formulation, and furthermore note that,
by explicitly computing integrals, the periods $\Omega_v(A/L)$
can be related to those obtained by integrating measures as occurring in the classical formulation of the
Birch and Swinnerton-Dyer conjecture -- see, for example, Gross~\cite[p. 224]{G-BSD}.

For the remainder of this section, we assume that $C(A, \QG)$ is valid. It is then easy to see that, for any order $\Lambda$ in $\IQ[G]$ that contains $\IZ[G]$, the
validity of C$_{p,j}(A,\Lambda)$ is independent of the choice of isomorphism $j:\bc\to\bc_p$, and so we fix such a $j$ for the remainder of this section.
In fact, all relevant elements of $\bc[G]$ appearing in the statements of our results will actually belong to $\bq[G]$
(as a consequence of an easy application of Theorem \ref{rat theorem}) and so we will consider them simultaneously as elements of $\bq_p[G]\subset\bc_p[G]$
in the natural way without any explicit mention of $j$.

Let $\mathcal{M}$ denote the maximal $\ZZ$-order in
$\IQ[G]$. For any $\psi\in\widehat{G}$, let $\mathcal{O}_{\psi}$
be the valuation ring of $\IQ_p(\psi)$. Let $\frp_\psi$
be the (unique) prime ideal of $\mathcal{O}_{\psi}$ above $p$.
We write $v_{\frp_\psi}$ for the normalised valuation defined by $\frp_\psi$.

\mpar{max theorem}
\begin{thm}\label{max theorem}
Let $\calP$ and $\calP^t$ be any choice of points such that (\ref{global points}) holds. We assume that $\Sha_p(A_F) = 0$. Then the following are equivalent.

\begin{itemize}
\item[(i)] $C_p(A, \calM)$ is valid.
\item[(ii)] BSD$_p(L)$ is valid for all intermediate fields $L$ of $F/k$.
\item[(iii)] For each $\psi \in \widehat{G}$ one has
\[
v_{\frp_\psi} \left( \frac{\mathcal{L}^*_{\psi}}{\lambda_{\psi}(\mathcal{P},\mathcal{P}^t)} \right) = b_\psi \text{ where }
b_\psi :=\sum_{s=0}^{t_\psi-1} p^s m_s.
\]
\item[(iv)] $$\sum_{\psi \in \widehat{G}} \frac{\calL_\psi^*}{\lambda_\psi(\calP, \calP^t) \delta_\psi} e_\psi\in \calM_p^\times.$$
\end{itemize}
\end{thm}

To describe the full range of implications of the validity of $C_p(A, \bz[G])$ requires yet more work and some further notations.

For each finite extension $L/k$ and natural number $n$ we write $\Sel^{(p^n)}(A_L)$ for the Selmer group associated
to the isogeny $[p^n]$. We define the $p$-primary Selmer group by
\[
 \Sel_p(A_L) := \directlim  \Sel^{(p^n)}(A_L).
\]
We recall that one then obtains a canonical short exact sequence
\[
0 \lra \Qp/\Zp \tensorZ A(F) \lra  \Sel_p(A_F) \lra \Sha_p(A_F) \lra 0
\]
of $\ZpG$-modules, from which upon taking Pontryagin duals one derives a canonical short exact sequence
\mpar{Sel ses}
\begin{equation}\label{Sel ses}
0 \lra \Sha_p(A_F)^\vee  \lra  \Sel_p(A_F)^\vee \lra A(F)_p^* \lra 0.
\end{equation}
We will throughout use this canonical short exact sequence to fix identifications of $\left( \Sel_p(A_F)^\vee \right)_{{\rm tor}}$ with $ \Sha_p(A_F)^\vee $ and
of $\left( \Sel_p(A_F)^\vee \right)_{{\rm tf}}$ with $A(F)_p^*$.

In \cite{bmw} a suitable integral model $R\Gamma_f(k, T_{p, F}(A))$ of the finite support cohomology of Bloch and
Kato for the base change through $F/k$ of the $p$-adic Tate module of $A^t$ is defined and then used in order to define
an `equivariant regulator' which is essential to the explicit reformulation of C$_p(A,\bz[G])$
(see \cite[Th.~5.1]{bmw}). We will recall this reformulation in Section \ref{bmw reformulation}.

By \cite[Lem. 4.1]{bmw}, $R\Gamma_f(k,T_{p,F}(A))$ is under our current hypotheses a perfect complex of
$\bz_p[G]$-modules which is acyclic outside degrees 1 and 2 and whose cohomology groups in degrees 1 and 2 canonically
identify with $A^t(F)_p$ and ${\rm Sel}_p(A_F)^\vee$ respectively. 
We recall that given any complex $E$ with just two nonzero cohomology modules $H^m(E)$ and $H^n(E)$, $n > m$, the complex
$\tau^{\le n}\tau^{\ge m} E$ represents an element in the Yoneda ext-group $\Ext_{\ZpG}^{n-m+1}(H^n(E), H^m(E))$. Here $\tau$ is the truncation of 
complexes preserving cohomology in the indicated degrees. In this way, $R\Gamma_f(k,T_{p,F}(A))$ uniquely determines a class $\delta_{A,K,p}$ in 
${\rm Ext}^2_{\bz_p[G]}({\rm Sel}_p(A_F)^\vee,A^t(F)_p)$. The element $\delta_{A,K,p}$ is furthermore
perfect, meaning that it can be represented as a Yoneda 2-extension by a four term exact sequence in which each of the two middle modules is perfect when considered
as an object of $D(\Zp[G])$.
We will use Proposition \ref{MW iso}
to fix an explicit $2$-syzygy of the form
\mpar{inexplicitsyzygy}
\begin{equation}\label{inexplicitsyzygy}
0\to M\stackrel{\iota}{\to}F^0\to F^1\to A(F)_p^*\to 0,
\end{equation}
in which we set \[M:=\bigoplus_{(t,j)} \Zp[G/H_t]\] and both $F^0$ and $F^1$ are finitely generated free $\bz_p[G]$-modules
and then use the exact sequence (\ref{inexplicitsyzygy}) to compute $\Ext^2_{\ZpG}(A(F)_p^*, A^t(F)_p)$ via the 
canonical isomorphism
\[
\Ext^2_{\ZpG}(A(F)_p^*, A^t(F)_p) \simeq \Hom_{\ZpG}(M, A^t(F)_p) / \iota_* \left(  \Hom_{\ZpG}(F^0, A^t(F)_p) \right).
\]

If we now assume that $\Sha_p(A_F)$ vanishes, we may identify $\Sel_p(A_F)^\vee$ and $A(F)^*_p$, so that $\delta_{A,F,p}$
uniquely determines an element of the above quotient.  We will prove (see Lemmas \ref{extvanishes} and \ref{ext iso} below)
that we may choose a representative $\Phi\in\Hom_{\ZpG} (M, A^t(F)_p)$ of $\delta_{A,F,p}$ with the following properties:
\begin{itemize}
\item [(P1)] $\Phi$ is bijective,
\item [(P2)]  for every $j\in[m_n]$, $\Phi$ restricts to send an element  $x_{(n,j)}$  of the $(n,j)$-th direct summand $\bz_p[G]$ to $x_{(n,j)}P^t_{(n,j)}$.
\end{itemize}

For a fixed choice of points $\calP$ and $\calP^t$ such that (\ref{global points}) holds and of $\Phi\in\Hom_{\ZpG} (M, A^t(F)_p)$ as above,
we fix a canonical $\bz_p[G/H_t]$-basis element $e_{(t,j)}$
of each direct summand $\bz_p[G/H_t]$ of $M$ and fix any elements $\Phi_{(t,j), (s,i)}$ of $\Zp[G]$ with the property that
\mpar{definingphi}
\begin{equation}
\label{definingphi}
\Phi(e_{(s,i)}) = \sum_{(t,j)} \Phi_{(t,j), (s,i)} P^t_{(t,j)}.
\end{equation}
We thus obtain an invertible  matrix $\left( \Phi_{(t,j), (s,i)} \right)_{(t,j), (s,i)}$ with entries in $\bz_p[G]$, which by abuse of notation
we shall also denote by $\Phi$. The matrix $\Phi$ is of the form
\mpar{IMPROVE}
\mpar{Phi mat}
\begin{equation}\label{Phi mat}
\left(
    \begin{array}{ccc}
      \left(\Phi_{(t,j),(s,i)} \right)_{t,s < n} & \vline &
      \begin{array}{c}
        0\\\vdots\\0
      \end{array}
\\
      \hline \\
      \begin{array}{ccc} 0 & \hdots & 0 \end{array} &\vline & I_{m_n}
\end{array}
\right)
\end{equation} with $I_{m_n}$ denoting the identity $m_n\times m_n$ matrix.

Recall the definition of $t_\psi$ in Definition \ref{minor def}.
We define the 'lower $\psi$-minor' of $\Phi$ by setting
\[
  \varepsilon_\psi(\Phi) := \det\left( \psi \left( \Phi_{t_\psi} \right) \right).
\]
We note firstly that, since the chosen points $P^t_{(t,j)}$ satisfy (\ref{global points}), each element $\varepsilon_\psi(\Phi)$
(and, indeed, even the matrix $\psi \left( \Phi_{t_\psi} \right))$ is independent of our particular
choice of elements $\Phi_{(t,j), (s,i)} \in \Zp[G]$ with the property that (\ref{definingphi}) holds.

\mpar{ZpG theorem}
\begin{thm}\label{ZpG theorem}
Let $\calP$ and $\calP^t$ be any choice of points such that (\ref{global points}) holds. Assume that $\Sha_p(A_F) = 0$.
Let $\Phi \in \Hom_{\ZpG}(M, A^t(F)_p)$ be any representative of $\delta_{A,F,p}$ such that (P1) and (P2) hold.
Then $C_p(A, \ZG)$ is valid if and only if
\mpar{ZpG criterion}
\begin{equation}\label{ZpG criterion}
\sum_{\psi \in \widehat G} \frac{ \calL_\psi^*  }{ \lambda_\psi(\calP, \calP^t) \cdot \varepsilon_{\psi}(\Phi) \cdot \delta_\psi  }e_\psi \in \ZpG^\times.
\end{equation}
\end{thm}

\begin{remark} Theorem \ref{ZpG theorem} can be reformulated in terms of explicit congruences.
\end{remark}



Via Theorem \ref{ZpG theorem}, we now obtain completely explicit predictions concerning congruences in the augmentation filtration
of the integral group ring $\bz_p[G]$ for leading terms at $s=1$ of the relevant Hasse-Weil-$L$-functions of $A$ normalised by
our twisted regulators. We recall that such predictions constitute a refinement and generalisation of the congruences for modular
symbols that are conjectured by Mazur and Tate in \cite{mazurtate}.

In order to state such conjectural congruences, we require the following notation: if the inequality ${\rm rk}(A(F^J))\leq
|G/J|{\rm rk}(A(k))$ of Corollary \ref{ranks} is strict for some subgroup $J$ of $G$, we may and will denote by $H=H_{t_0}$
the smallest non-trivial subgroup of $G$ with the property that $m_H \neq 0$. Hence $t_0$ is the maximal index with the
properties $m_{t_0} \ne 0$ and $t_0 < n$. We then define
\[
\mathcal{L}:=
\begin{cases}
\sum\limits_{\psi\in\widehat{G}} \frac{ \calL_\psi^*  }{ {\rm det}(\psi(R(\calP, \calP^t))) }e_\psi, & \text{ if } {\rm rk}(A(F^J))=
|G/J|{\rm rk}(A(k)) \text{ for every } J, \\
\sum\limits_{\psi|_H \neq 1} \frac{ \calL_\psi^*  }{ \lambda_\psi(\calP, \calP^t) }e_\psi, & \text{ otherwise}.
\end{cases}
\]

We also let $I_{G,p}$ denote the kernel of the augmentation map $\ZpG \lra \Zp$.

\mpar{Cor 1}
\begin{cor}\label{Cor 1}
  Let $\calP$ and $\calP^t$ be any choice of points such that (\ref{global points}) holds. Assume that $\Sha_p(A_F) = 0$.
Let $\Phi \in \Hom_{\ZpG}(M, A^t(F)_p)$ be any representative of $\delta_{A,F,p}$ such that (P1) and (P2) hold.
If $C_p(A, \ZG)$ is valid, then
\begin{itemize}
\item [(i)] $\calL$ belongs to the ideal $I_{G,p}^h$ of $\bz_p[G]$, where $h := \sum_{t < n} m_t$.
\item [(ii)] $\epsilon := \det \left( {\trivchar}(\Phi) \right) \in \Ze_p^\times$.
\item [(iii)] $v := (-1)^{d\cdot |S_r|}\frac{L^*_{S_r}(A/k,1)\cdot(\sqrt{|d_k|})^d}{\Omega(A)\cdot\det \left( {\trivchar}(R(\calP, \calP^t)) \right) } \in \Ze_p^\times$.
\item [(iv)] $\calL \equiv \frac v \epsilon \cdot \prod_{t < n} \left( \sigma^{p^t} - 1 \right)^{m_t} (\mod I_{G, p}^{h+1}).$
\end{itemize}
\end{cor}

The theory of organising matrices developed by Burns and the second named author in \cite{omac} allows one to derive the
containment $\calL\in I_{G,p}^h$ of Corollary \ref{Cor 1}(i) from the assumed validity of conjecture $C_p(A, \ZG)$ in
situations in which $\Sha_p(A_F)$ is non-trivial. In this greater level of generality, it furthermore leads to explicit statements
concerning annihilation of Tate-Shafarevich groups and (generalised) `strong main conjectures' of the kind that Mazur and Tate
ask for in \cite[Remark after Conj. 3]{mazurtate}. Namely, we obtain the following result:

\mpar{Ann theorem}
\begin{thm}\label{Ann theorem}
  Let $\calP$ and $\calP^t$ be any choice of points such that (\ref{global points}) holds. If $C_p(A, \ZG)$ is valid, then
\begin{itemize}
\item [(i)] $\calL$ belongs to the ideal $I_{G,p}^h$ of $\bz_p[G]$, where $h := \sum_{t < n} m_t$.
\item [(ii)] $\calL$ annihilates the $\bz_p[G]$-module $\Sha_p(A_F^t)$.
\item [(iii)] There exists a (finitely generated) free $\bz_p[G]$-submodule $\Pi$ of $\Sel_p(A_F)^\vee$ of (maximal) rank $m_n$
with the property that $\calL$ generates the Fitting ideal of the quotient $\Sel_p(A_F)^\vee/\Pi$.
\end{itemize}
\end{thm}

\begin{remark}\label{shaorpermutation} It will become clear in the course of the proof that, provided that there exist sets of points $\calP$ and $\calP^t$
such that (\ref{global points}) holds from which one may construct the element $\calL$, Theorem \ref{Ann theorem} remains valid even if hypothesis (h) fails to hold. This
fact is relevant because, as we will see in Section \ref{examples}, it allows us to obtain numerical supporting evidence for $C_p(A, \ZG)$ (via verifying the
explicit assertions of Theorem \ref{Ann theorem}) in a wider range of situations.
\end{remark}

Let $\# \colon \ZpG \lra \ZpG$ denote the involution induced by $g \mapsto g^{-1}$. Recalling that the Cassels-Tate pairing
induces a canonical isomorphism between $\Sha_p(A_F)^\vee$ and  $\Sha_p(A^t_F)$, we immediately obtain the following corollary:

\mpar{Cor 2}
\begin{cor}\label{Cor 2}
Under the assumptions of Theorem \ref{Ann theorem} one has that the element $\calL^\#$ of $I_{G,p}^h$ annihilates
the $\bz_p[G]$-module $\Sha_p(A_F)$.
\end{cor}

\mpar{bmw reformulation}
\section{An explicit reformulation of conjecture C$_p(A,\bz[G])$}\label{bmw reformulation}

\mpar{preliminaries}
\subsection{K-theory and refined Euler characteristics}\label{preliminaries}

Let $R$ be either $\Ze$ or $\Zp$ and, for the moment, let $G$ be any finite group. We write $K$ for the quotient field
of $R$ and let $\mathbb{E}$ be a field extension of $K$. Let $\Lambda$ be an $R$-order in $K[G]$. We recall that there is a
canonical exact sequence of algebraic $K$-groups
\mpar{K seq}
\begin{equation}\label{K seq}
K_1(\Lambda) \lra K_1(\mathbb{E}[G]) \stackrel{\partial^1_{\Lambda, \mathbb{E}}}\lra K_0(\Lambda, \EG) \lra K_0(\Lambda) \lra
 K_0(\mathbb{E}[G])
\end{equation}
where $K_0(\Lambda, \EG)$ is the relative algebraic $K$-group, as defined by Swan in \cite[p.~215]{swan}, associated to the ring inclusion $\Lambda\subseteq\EG$.

For any ring $\Sigma$ we write $\zeta(\Sigma)$ for its center.
We let $\nr_{\EG} \colon K_1(\EG) \lra \zeta(\EG)^\times$ denote the (injective) homomorphism induced by the reduced norm map.
If $\Lambda$ is a $\Ze$-order in $\QG$ we write
\begin{eqnarray*}
  \delta_G &\colon& \zeta(\RG)^\times \lra K_0(\Lambda, \RG), \\
  \delta_{G,p} &\colon& \zeta(\CpG)^\times \lra K_0(\Lambda_p, \CpG)
\end{eqnarray*}
for the extended boundary homomorphisms as defined in \cite[Sec.~4.2]{bufl01}.
Recall that
\[
\delta_G \circ \nr_\RG = \partial^1_{\Lambda, \Re}, \quad \delta_{G,p} \circ \nr_{\CpG} = \partial^1_{\Lambda_p, \Cp}.
\]
By the general construction described in \cite[Prop.~2.5]{bufl01} (and \cite[Lem.~5.1]{additivity}) each pair
$(C^\bullet, \lambda)$ consisting of a complex $C^\bullet \in D^p(\Lambda_p)$ and an isomorphism of $\CpG$-modules
$$
\lambda \colon  \Cp \otimes_{\Zp}  \left( \bigoplus_{i\in\Ze}H^{2i}(C^\bullet)\right) \lra \Cp \otimes_{\Zp} \left( \bigoplus_{i\in\Ze}H^{2i+1}(C^\bullet) \right)
$$ 
gives rise to a refined Euler characteristic $\chi_{G,p}(C^\bullet, \lambda)
\in K_0(\Lambda_p, \CpG)$. For an explicit example of the computation of  $\chi_{G,p}(C^\bullet, \lambda)$ in a special case,
which is also relevant for the computations in this paper, we refer the reader to \cite[Sec.~3]{bleytwo}.

It is well known that $\partial^1_{\Lambda_p, \Cp}$ is onto and that $\nr_{\CpG}$ is an isomorphism.
We therefore deduce from (\ref{K seq}) that
\mpar{K iso}
\begin{equation}\label{K iso}
K_0(\Lambda_p, \CpG) \simeq \zeta(\CpG)^\times / \nr_{\CpG} \left( K_1(\Lambda_p) \right).
\end{equation}
Since $\Lambda_p$ is semilocal, we can replace $K_1(\Lambda_p)$ by $\Lambda_p^\times$ in (\ref{K iso}). Moreover, it follows from (\ref{K iso}) that for
an element $\xi \in \zeta(\CpG)^\times$ one has that
$\delta_{G, p}(\xi) = 0$ if and only if $\xi \in \nr_{\CpG} \left( \Lambda_p^\times \right)$.
Finally, if $G$ is abelian, we have that
\[
K_0(\Lambda_p, \CpG) \simeq \CpG^\times / \Lambda_p^\times,
\]
and hence $\delta_{G, p}(\xi) = 0$ if and only if $\xi \in \Lambda_p^\times$.

In this context we also recall \cite[Lem.~2.5]{bleyone}. We naturally interpret $K_0(\Lambda, \bq[G])$ and $K_0(\Lambda_p, \QpG)$ as subgroups of $K_0(\Lambda, \br[G])$ and $K_0(\Lambda_p, \CpG)$
respectively, and recall that if $\xi \in \zeta(\br[G])^\times$, then
\[
\delta_{G}(\xi) \in K_0(\Lambda, \bq[G]) \iff \xi \in \zeta(\bq[G])^\times
\]
while if $\xi \in \zeta(\CpG)^\times$, then
\[
\delta_{G,p}(\xi) \in K_0(\Lambda_p, \QpG) \iff \xi \in \zeta(\QpG)^\times.
\]

We finally recall that, for any isomorphism $j: \CC\cong \CC_p$ , there is an induced composite homomorphism of abelian groups
\[ j_{G,*}: K_0\bigl(\Lambda,\RR[G]\bigr) \to K_0\bigl(\Lambda,\CC[G]\bigr)\cong K_0\bigl(\Lambda,\CC_p[G]\bigr) \to K_0\bigl(\Lambda_p,\CC_p[G]\bigr)\]
(where the first and third arrows are induced by the inclusions $\RR[G]\subset \CC[G]$ and $\Lambda\subset \Lambda_p$ respectively). We also write
$j_*:\zeta(\bc[G])^\times\to \zeta(\bc_p[G])^\times$ for the obvious map induced by $j$, and note that it is straightforward to check that one has
\[j_{G,*}\circ\delta_G=\delta_{G,p}\circ j_*.\]

\mpar{relevant bmw results}
\subsection{Relevant results from \cite{bmw}}\label{relevant bmw results}



Conjecture C$(A,\bz[G])$ is formulated as the vanishing of the `equivariant Tamagawa number' $T\Omega\bigl(h^1(A_{F})(1),\ZZ[G]\bigr)$
of $K_0(\bz[G],\br[G])$ that is defined in \cite[Conj. 4]{bufl01} and constructed (unconditionally under the assumed validity of hypothesis (g)) via the formalism
of virtual objects from the various canonical comparison
morphisms between the relevant realisations and cohomology spaces associated to the motive $h^1(A_{F})(1)$ (for more details see \cite{bufl01}). 

Motivated by work of Bloch and Kato, and in order to isolate the main arithmetic difficulties involved in making $T\Omega\bigl(h^1(A_{F})(1),\ZZ[G]\bigr)$
explicit, the approach of \cite{bmw} relies upon the definition of a suitable (global) finite support cohomology complex $\Cf:=R\Gamma_f(k, \TpFA)$
(see \cite[Sec.~4.2]{bmw}).
Under the hypotheses of \cite{bmw} the complex $\Cf$ is perfect and acyclic
outside degrees one and two. Moreover, there are canonical identifications of $H^1(\Cf)$ and $H^2(\Cf)$
with $A^t(F)_p$ and $\Sel_p(A_F)^\vee$ respectively (see \cite[Lemma 4.1]{bmw}). Hence, for a given isomorphism $j:\bc\to\bc_p$, the $\Ce$-linear extension
of the N\'eron-Tate height pairing of $A$ defined relative to the field $F$ induces a canonical trivialisation
\begin{multline*}
\lambda^{{\rm NT},j}_{A,F}\colon \CC_p\otimes_{\ZZ_p} H^1\bigl(C^{f,\bullet}_{A,F}\bigr) \cong \CC_p\otimes_{\ZZ_p}A^t(F)_p \\\cong \CC_p\otimes_{\CC,j}(\CC\otimes_\ZZ A^t(F))
 \cong \CC_p\otimes_{\CC,j}\Hom_{\CC}\bigl(\CC\otimes_{\ZZ}A(F),\CC\bigr)\\ \cong \CC_p\otimes_{\ZZ_p}\Hom_{\ZZ_p}\bigl(A(F)_p,\ZZ_p\bigr) \cong \CC_p\otimes_{\ZZ_p}H^2\bigl(C^{f,\bullet}_{A,F}\bigr).\end{multline*}
It is finally proved in \cite[Th. 5.1]{bmw} that 
\begin{equation}\label{tomega}j_{G,*}\bigl(T\Omega\bigl(h^1(A_{/F})(1),\ZZ[G]\bigr)\bigr)=\delta_{G,p}\left( j_*(\calL^*_{A, F/k}) \right)+\chi_{G,p}\left(\Cf,(\lambda^{{\rm NT},j}_{A,F})^{-1} \right),\end{equation}
or equivalently that conjecture C$_{p,j}(A,\bz[G])$ is valid if and only if

\mpar{reform eq}
\begin{equation}\label{reform eq}
\delta_{G,p}\left( j_*(\calL^*_{A, F/k}) \right) = -\chi_{G,p}\left(\Cf,(\lambda^{{\rm NT},j}_{A,F})^{-1} \right).
\end{equation}

In order to prove our results stated in Section \ref{statements} we must therefore compute the refined Euler characteristic
$\chi_{G,p}\left(\Cf,(\lambda^{{\rm NT},j}_{A,F})^{-1} \right)$ in terms of the heights of the chosen sets of points $\calP$ and $\calP^t$.

\section{The proofs}\label{Proofs}

\mpar{Choice of points}
\subsection{The proof of Proposition \ref{MW iso}}\label{Choice of points}

In this subsection we will prove Proposition \ref{MW iso}. The existence of global points $\Ptj$ and $\Pttj$
such that (\ref{global points}) holds is then an immediate consequence of Roiter's lemma (see \cite[(31.6)]{curtisr}).


To ease notation we set $H^1 := H^1(\Cf)=A^t(F)_p$ and $H^2 := H^2(\Cf)=\Sel_p(A_F)^\vee$. We recall that, for any intermediate field $L$ of $F/k$, we may and will use the
relevant canonical short exact sequence of the form (\ref{Sel ses}) to identify $(\Sel_p(A_L)^\vee)_{{\rm tor}}$ with $\Sha_p(A_L)^\vee$ and
$(\Sel_p(A_L)^\vee)_{{\rm tf}}$ with $A(L)_p^*$. 

Under the assumed validity of hypotheses (a)-(e), the result of \cite[Prop.~3.1]{bmwselmer} directly combines with hypothesis (h) to imply that,
for every non-trivial subgroup $J$ of $G$, the Tate cohomology group
$\widehat{H}^{-1}( J, H^2_{{\rm tf}})$ vanishes and the module $(H^2)_J$ is torsionfree.
By the definition of Tate cohomology,
we have that the finite group $\widehat{H}^{-1}( J, H^2)$ identifies with a submodule of $(H^2)_J$ and therefore vanishes too.



Furthermore, since the complex $\Cf$ is perfect and acyclic outside degrees 1 and 2, for each subgroup $J$ of $G$ the group $\widehat{H}^{1}( J, H^1)$ is
isomorphic to $\widehat{H}^{-1}( J, H^2)$. In addition, since $G$ is cyclic, the Tate cohomology of each $J$ is periodic of order 2 and so
$\widehat{H}^{-1}( J, H^1)$ also vanishes.

We next note that, since $G$ is a $p$-group, hypothesis (a) implies that $A^t(F)_p=H^1$ is torsion-free.

We now apply the main result \cite[Th. 2.4]{yakovlev} of Yakovlev to see that both $A^t(F)_p=H^1$ and $A(F)_p^*=H^2_{{\rm tf}}$ are $\ZpG$-permutation
modules, that is, that there exist isomorphisms of the form
\[
A^t(F)_p \simeq \bigoplus_{J\leq G}\IZ_p[G/J]^{<r_J>}, \quad A(F)_p^* \simeq \bigoplus_{J\leq G}\IZ_p[G/J]^{<s_J>}
\]
for some sets of non-negative integers $\{r_J\}$ and $\{s_J\}$.
But the N\'eron-Tate height pairing induces an isomorphism of $\CpG$-modules between $\Cp \tensorZp A^t(F)_p$ and
$\Cp \tensorZp A(F)_p^*$ and so by rank considerations we find that $r_J=s_J=:m_J$ for every $J$.
Finally, it is easy to see that the $\Zp$-linear dual of a permutation module is again a permutation module of the same
form. Therefore the canonical isomorphism $A(F)_p^{**} \simeq A(F)_p$ shows that one also has that
\[
A(F)_p \simeq \bigoplus_{J\leq G}\IZ_p[G/J]^{<m_J>}.
\]
\qed




\mpar{Proof of ZpG theorem}
\subsection{The proof of Theorem \ref{ZpG theorem}}\label{Proof of ZpG theorem}

Recall that in addition to our running hypothesis (a) - (h) we also assume that $\Sha_p(A_F) = 0$. In particular,
we identify $\Sel_p(A_F)^\vee$ with $A(F)_p^*$ via the canonical map in (\ref{Sel ses}).

We fix an isomorphism of fields $j:\bc\to\bc_p$. From (\ref{reform eq}) and the discussion in \S \ref{preliminaries} it is clear that it will be enough to show that
\mpar{chiGp eq}
\begin{equation}
  \label{chiGp eq}
-\chi_{G,p}\left(\Cf,(\lambda^{{\rm NT},j}_{A,F})^{-1} \right) =
\delta_{G,p} \left( \sum_{\psi \in \widehat{G}} j(\lambda_\psi(\calP, \calP^t)) \epsilon_\psi(\Phi)  \delta_\psi  e_\psi \right)
\end{equation} (we recall that, since we have assumed the validity of C$(A,\bq[G])$, one actually has that the validity of C$_{p,j}(A,\bz[G])$ is equivalent to the
validity of C$_{p}(A,\bz[G])$).

We begin by defining, for every pair $(s,i)$, an element $P^*_{(s,i)} \in A(F)_p^*$ by setting, for every pair $(t,j)$ and element $\tau$ of $G$,
\begin{equation}\label{dualpoints}
P^*_{(s,i)} (\tau P_{(t,j)}) =
\begin{cases}
  1, & \text{ if } s=t, i=j \text{ and } \tau \in H_s \\
  0, & \text{ otherwise.}
\end{cases}
\end{equation}

\mpar{dual basis}
\begin{lem}\label{dual basis}
$A(F)_p^* = \bigoplus_{(s,i)} \Zp[G/H_s] P_{(s,i)}^*$ with each summand $\Zp[G/H_s] P_{(s,i)}^*$ isomorphic to $\Zp[G/H_s]$.
\end{lem}

\begin{proof}
If $\gamma \in G$, then
\mpar{dual basis eq}
\begin{eqnarray}\label{dual basis eq}
&&\left( \gamma P_{(s,i)}^* \right) \left( \tau P_{(t,j)} \right) =  P_{(s,i)}^*(\gamma^{-1}\tau  P_{(t,j)} ) = 1 \nonumber \\
&\iff& s=t, i=j \text{ and } \gamma \equiv \tau (\mod H_s). 
\end{eqnarray}
Hence we have $ \gamma P_{(s,i)}^* =  P_{(s,i)}^* $ for $\gamma \in H_s$. Moreover, it easily follows that the
maps $ \gamma P_{(s,i)}^*$ with $\gamma \in G/H_s$ form a $\Zp$-basis of $A(F)_p^*$ (actually the $\Zp$-dual basis of
$\tau P_{(t,j)}$ with $\tau \in G/H_t$).
\end{proof}

We now proceed to fix an explicit $2$-syzygy of the form (\ref{inexplicitsyzygy}).  For this purpose, we first recall
that $H_t = \langle \sigma^{p^t} \rangle$. For each pair $(t,j)$ corresponding to the subgroup $H_t$ of $G$
and $j\in [m_t]$ we hence have a $2$-extension
\[
0 \lra \Zp[G/H_t] \stackrel{\iota_{t}}\lra \ZpG \stackrel{\sigma^{p^t} - 1}\lra \ZpG \stackrel{\pi_{t,j}}\lra \Zp[G/H_t] P^*_{(t,j)} \lra 0.
\]
In this sequence we let $\iota_t$ denote the (well-defined) map which sends the image of
an element $x\in\bz_p[G]$ under the natural surjection $\bz_p[G]\to\bz_p[G/H_t]$ to the element ${\rm Tr}_{H_t}x$ of $\bz_p[G]$,
while $\pi_{t,j}$ sends the element 1 of $\ZpG$ to the element $P^*_{(t,j)}\in A(F)^*_p$ defined in (\ref{dualpoints}).
Lemma \ref{dual basis} then implies that, summing over all pairs $(t,j)$ we obtain a $2$-extension
\mpar{syzygy}
\[
  0 \lra M \stackrel{\iota}\lra F^0 \stackrel{\Theta}\lra F^1 \stackrel{\pi}\lra A(F)^*_p \lra 0.
\]
with
\begin{eqnarray*}
  F^0 = F^1 = X &:=& \bigoplus_{(t,j)} \ZpG, \\
             M &:=&  \bigoplus_{(t,j)} \Zp[G/H_t].
\end{eqnarray*}

We now recall that we have a canonical isomorphism
\[
\Ext_{\ZpG}^2(A(F)_p^*, A^t(F)_p) \simeq \Hom_{\ZpG} (M, A^t(F)_p) / \iota_*( \Hom_{\ZpG} (F^0, A^t(F)_p) )
\]
under which an element $\phi$ of $\Hom_{\ZpG} (M, A^t(F)_p)$ corresponds to the element $\epsilon(\phi)$ of
$\Ext_{\ZpG}^2(A(F)_p^*, A^t(F)_p)$ which has the bottom row of the commutative diagram with exact rows

\begin{equation}\label{pushout}
\begin{CD} 0 @>  >> M @> \iota >> X @> \Theta >> X @> \pi >> A(F)^*_p @>  >> 0\\
@. @V \phi VV  @V VV  @\vert @\vert @. \\
0 @>  >> A^t(F)_p @>  >> X(\phi) @>  >> F^1 @> \pi >> A(F)^*_p @>  >> 0,
\end{CD}
\end{equation}
as a representative. In this diagram $X(\phi)$ is defined as the push-out of $\iota$ and $\phi$.
We now proceed to prove that, when considering perfect elements $\epsilon(\phi)$ of $\Ext_{\ZpG}^2(A(F)_p^*, A^t(F)_p)$,
one may without loss of generality restrict attention to a special class of elements $\phi$ of $\Hom_{\ZpG} (M, A^t(F)_p)$.

\mpar{extvanishes}
\begin{lem}\label{extvanishes}
For all subgroups $J$ of $G$ one has
\begin{itemize}
\item [(i)] $\Ext_{\ZpG}^2(\ZpG, \Zp[G/J]) = 0$,
\item [(ii)] $\Ext_{\ZpG}^2(\Zp[G/J], \Zp[G]) = 0$
\end{itemize}
\end{lem}

\begin{proof} Claim (i) is clear. Concerning claim (ii), we first note that since $\Zp[G/J]$ is
$\bz_p$-torsion-free, there is an isomorphism of the form
\[
{\rm Ext}^{2}_{\bz_p[G]}(\bz_p[G/J],\bz_p[G]) \cong H^{2}(G,{\rm Hom}_{\bz_p}(\bz_p[G/J],\bz_p[G])).
\]
But the $G$-module ${\rm Hom}_{\bz_p}(\bz_p[G/J],\bz_p[G])$ is cohomologically-trivial and therefore ${\rm Ext}^{2}_{\bz_p[G]}(\bz_p[G/J],\bz_p[G])$ vanishes, as required.
%
\end{proof}

Lemma \ref{extvanishes} now implies that we can without loss of generality restrict attention to those
elements $\phi$ of $\Hom_{\ZpG} (M, A^t(F)_p)$ which satisfy (P2) and, in addition,
by the argument of \cite[Lemma~4.3]{bleytwo}, which are furthermore injective.

\mpar{ext iso}
\begin{lem}\label{ext iso}
Suppose that $\phi\in\Hom_{\ZpG} (M, A^t(F)_p)$ has all of the properties described in the previous paragraph. Then the
element $\epsilon(\phi)$ of ${ \rm Ext}^2_{\bz_p[G]}(A(F)_p^*,A^t(F)_p)$ is perfect if
and only if $\phi$ is an isomorphism.
\end{lem}

\begin{proof}
The fact that $\phi$ restricts to send an element $x_{(n,j)}$ of the $(n,j)$-th direct summand $\bz_p[G]$ to $x_{(n,j)}P^t_{(n,j)}$ immediately implies that $\cok(\phi) = \cok(\phi')$ where
\[
\phi': \bigoplus_{(t,j), t < n} \Zp[G/H_t]\to\bigoplus_{(t,j), t < n} \Zp[G/H_t]P^t_{(t,j)}
\]
is the map obtained by restriction of $\phi$. Since $\phi$ is injective, the commutative diagram (\ref{pushout}) implies that the $2$-extension $\epsilon(\phi)$
is perfect if and only $\cok(\phi) = \cok(\phi')$ is cohomologically trivial. Note that  $H_{n-1}$ clearly acts trivially on $\cok(\phi')$.
So, if $\cok(\phi')$ is cohomologically trivial, then
\[
\cok(\phi') / p\cok(\phi') =  \widehat H^0 (H_{n-1}, \cok(\phi')) =  0.
\]
It then follows that $\cok(\phi')$ must itself vanish, as required.
\end{proof}

We henceforth fix $\Phi \in \Hom_{\ZpG}(M, A^t(F)_p)$ representing the element $\delta_{A,F,p} \in \Ext_{\ZpG}^2( \Sel_p(A_F)^\vee, A^t(F)_p)$
which is specified by $\Cf$. Recall that by our current assumption $\Sha_p(A_F) = 0$ we identify
$\Sel_p(A_F)^\vee$ and $A(F)_p^*$. By Lemmas \ref{extvanishes} and \ref{ext iso} we may and will assume that $\Phi$ is an
isomorphism and furthermore that the matrix defined  in (\ref{definingphi}) is of the form (\ref{Phi mat}).

Having justified our choice of homomorphism $\Phi$, we now proceed to compute the term $-\chi_{G,p}\left(\Cf,(\lambda^{{\rm NT},j}_{A,F})^{-1} \right)$ that occurs
in (\ref{chiGp eq}) via a generalisation of the computations done in \cite[Sec.~4]{bleytwo}. For brevity, given any $\bz_p[G]$-module $N$, resp.
$\bz_p[G]$-homomorphism $h$, we set $N_{\bc_p}:=\bc_p\otimes_{\bz_p}N$, resp. $h_{\bc_p}:=\bc_p\otimes_{\bz_p}h$.

For any choice of respective splittings \[s_1:X_{\bc_p}\to M_{\bc_p}\oplus{\rm im}(\Theta)_{\bc_p}\] and
\[s_2:X_{\bc_p}\to{\rm ker}(\pi)_{\bc_p}\oplus A(F)^*_{\bc_p}\] of the short exact sequences induced by scalar extension of
\[ 0\to M\stackrel{\iota}\lra X\stackrel{\Theta}\lra{\rm im}(\Theta)\to 0\] and \[ 0\to{\rm ker}(\pi)\to X\stackrel{\pi}\lra A(F)_p^*\to 0\] respectively,
we write $\langle \lambda^{{\rm NT},j}_{A,F}\circ\Phi_{\bc_p},\Theta,s_1,s_2\rangle$ for the composite $\bc_p[G]$-automorphism of $X_{\bc_p}$ given by

\begin{eqnarray*}
X_\Cp &\stackrel{s_1}\lra& M_{\bc_p}\oplus{\rm im}(\Theta)_{\bc_p} \\
&\stackrel{(\Phi_{\bc_p}, id)}\lra& A^t(F)_{\bc_p}\oplus{\rm im}(\Theta)_{\bc_p} \\
&\stackrel{(\lambda^{{\rm NT},j}_{A,F}, id)}\lra& A(F)^*_{\bc_p}\oplus{\rm im}(\Theta)_{\bc_p} \\
&=& A(F)^*_{\bc_p}\oplus{\rm ker}(\pi)_{\bc_p} \\
&\stackrel{s_2^{-1}}\lra& X_\Cp.
\end{eqnarray*}

We also write $X^\bullet$ for the perfect complex of $\bz_p[G]$ modules $X\stackrel{\Theta}{\lra}X$ with the first term placed in degree 1 and the modules
$H^1(X^\bullet)$ and $H^2(X^\bullet)$ identified with $M$ and $A(F)_p^*$ respectively via the top row of diagram (\ref{pushout}).
By unwinding the definition of refined Euler characteristics and using their basic functoriality properties one then finds that,
independently of the choice of splittings $s_1$ and $s_2$, one has \begin{align*}-\chi_{G,p}\left(\Cf,(\lambda^{{\rm NT},j}_{A,F})^{-1} \right)=&
-\chi_{G,p}\left(X^\bullet,\Phi^{-1}_{\bc_p}\circ(\lambda^{{\rm NT},j}_{A,F})^{-1} \right)\\ =&\delta_{G,p}\left({\rm det}_{\bc_p[G]}(\langle \lambda^{{\rm NT},j}_{A,F}\circ\Phi_{\bc_p},\Theta,s_1,s_2\rangle)\right).\end{align*}

The proof of equality (\ref{chiGp eq}), and hence of Theorem \ref{ZpG theorem}, will thus be achieved by the following explicit computation.

\begin{prop}\label{thedeterminant} There exist splittings $s_1$ and $s_2$ as above with the property that
${\rm det}_{\bc_p[G]}(\langle \lambda^{{\rm NT},j}_{A,F}\circ\Phi_{\bc_p},\Theta,s_1,s_2\rangle)=\sum_{\psi \in \widehat{G}} j(\lambda_\psi(\calP, \calP^t)) \epsilon_\psi(\Phi)  \delta_\psi  e_\psi $.
\end{prop}

\begin{proof}
Let $\{ w_{(s,i)} : s=0, \ldots,n, i \in [m_s] \}$ be the standard basis of $X$. For each pair $(s,i)$ we write
$W_s = W_{(s,i)}$ for the kernel of the canonical map
\[
\Cp[G] \stackrel{}\lra \Cp[G/H_s],
\]
so that $W_s = ( \sigma^{p^s} - 1)\CpG = (1 - e_{H_s})\CpG$. We then have a commutative diagram
\mpar{split 2 extension}
\begin{equation}
  \label{split 2 extension}
\xymatrix{
0 \ar[r] & \Cp[G/H_s] \ar[r]^{\iota_s} & \CpG \ar[rr]^{\sigma^{p^s}-1} \ar@{>>}[rd] & & \CpG \ar[r]^{\pi_{s,i}\phantom{xxxxx}} & \Cp[G/H_s]P^*_{(s,i)} \ar[r] &  0 \\
&  &   & W_s \ar@{^{(}->}[ru] & &  &
}
\end{equation} with furthermore $\bigoplus_{(s,i)}W_{(s,i)}$ equal to ${\rm im}(\Theta)_{\bc_p}={\rm ker}(\pi)_{\bc_p}$.
We now fix the required splittings $s_1$ and $s_2$ by summing over all pairs $(s,i)$ the splittings of the short exact sequences in
(\ref{split 2 extension}) given by
\mpar{s1}
\begin{equation}
  \label{s1}
  \CpG \lra \Cp[G/H_s] \oplus W_s, \quad 1 \mapsto \left( \frac{1}{|H_s|}, \sigma^{p^s} - 1 \right)
\end{equation}
and 
\mpar{s2}
\begin{equation}
  \label{s2}
  \CpG \lra \Cp[G/H_s]P^*_{(s,i)} \oplus W_s, \quad 1 \mapsto \left(P^*_{(s,i)} , 1 - e_{H_s} \right)
\end{equation}
respectively. Note that for the inverse map in (\ref{s2}) we have $(P^*_{(s,i)}, 0) \mapsto e_{H_s}$ and
$(0, \sigma^{p^s} - 1) \mapsto  \sigma^{p^s} - 1$. 

After these preparations we proceed to compute the matrix $\Lambda^{\NT}(\Phi)$ which represents $\langle \lambda^{{\rm NT},j}_{A,F}\circ\Phi_{\bc_p},\Theta,s_1,s_2\rangle$
with respect to the fixed $\bc_p[G]$-basis $\{ w_{(s,i)} \}$ of $X_{\bc_p}$. From (\ref{s1}) and (\ref{definingphi}) it follows easily that the composite
of $s_1$ and $(\Phi_{\bc_p}, {\rm id})$ maps $w_{(s,i)}$ to
\[
\left( \frac{1}{|H_s|} \sum_{(t,j)} \Phi_{(t,j), (s,i)} P^t_{(t,j)}, \left( \ldots, \sigma^{p^s} - 1, \ldots \right) \right)
\]
in $A^t(F)_{\bc_p}\oplus{\rm im}(\Theta)_{\bc_p}=\left( \bigoplus_{(t,j)} \Cp[G/H_t] P^t_{(t,j)} \right) \oplus \left( \bigoplus_{(t,j)} W_t \right)$
with the only non-zero component in  $\oplus_{(t,j)} W_t$ at the
$(s,i)$-spot. By Lemma \ref{eval lambda NT} below this is further mapped by $(\lambda^{\NT,j}_{A,F}, {\rm id})$ to
\[
\left(  \frac{1}{|H_s|}  \sum_{(t,j)} \Phi_{(t,j), (s,i)} \sum_{(u,k)}
\left( \sum_{\tau \in G/H_u} j(\langle \tau P_{(u,k)}, P^t_{(t,j)} \rangle_F) \tau e_{H_u} \right) P^*_{(u,k)} ,
\left( \ldots, \sigma^{p^s} - 1, \ldots \right) \right).
\]
Rearranging the summation and applying the map $s_2^{-1}$ as described in (\ref{s2}) we obtain
\[
\sum_{(u,k)} \left( \frac{1}{|H_s|}  \sum_{(t,j)} \Phi_{(t,j), (s,i)}
\sum_{\tau \in G/H_u} j(\langle \tau P_{(u,k)},  P^t_{(t,j)} \rangle_F) \tau e_{H_u} \right) w_{(u,k)} + (\sigma^{p^s}-1) w_{(s,i)}.
\]
We now fix a character $\psi \in \widehat G$. We have that
\begin{eqnarray*}
  && \psi\left( \frac{1}{|H_s|}  \sum_{(t,j)} \Phi_{(t,j), (s,i)}
\sum\limits_{\tau \in G/H_u} j(\langle \tau P_{(u,k)}, P^t_{(t,j)} \rangle_F) \tau e_{H_u} \right) \\
&=&
\begin{cases}
  \frac{1}{|H_s|}  \sum\limits_{(t,j)} \Phi_{(t,j), (s,i)}
\sum\limits_{\tau \in G/H_u} j(\langle \tau P_{(u,k)}, P^t_{(t,j)}) \rangle_F \psi(\tau), & u \ge t_\psi, \\
0, & u < t_\psi,
\end{cases}
\end{eqnarray*}
while $\psi(\sigma^{p^s}-1)$ is equal to 0 if and only if $s \ge t_\psi$.

We immediately obtain that
\[
\det\left(\psi\left( \Lambda^\NT(\Phi) \right) \right) = j(\lambda_\psi(\calP, \calP^t)) \cdot \varepsilon_\psi(\Phi) \cdot \delta_\psi,
\]
as required.
\end{proof}

We finally provide the relevant Lemma used in the course of the above proof.

\mpar{eval lambda NT}
\begin{lem}\label{eval lambda NT}
$$\lambda^{\NT,j}_{A,F} \left(P^t_{(t,j)}\right) = \sum_{(u,k)} \left( \sum_{\tau \in G/H_u} j(\langle \tau P_{(u,k)}, P^t_{(t,j)} \rangle_F) \tau e_{H_u} \right) P^*_{(u,k)}.$$
\end{lem}

\begin{proof}
We recall that $\lambda^{\NT,j}_{A,F}$ is induced by $\langle \,,\, \rangle_F \colon A(F) \times A^t(F) \lra \CC$. For $P^t \in A^t(F)$
 we explicitly have $\lambda^{\NT,j}_{A,F}(P^t) = j(\langle\,\,\, , P^t \rangle_F)$. Let $f \in A(F)_p^*$ denote the map defined by the right hand side
of the equation in Lemma (\ref{eval lambda NT}). From (\ref{dual basis eq}) we immediately see that
$e_{H_u} P^*_{(u,k)} = P^*_{(u,k)}$. For each pair $(v,l)$ and $\gamma \in G/H_v$ we hence have that
\begin{eqnarray*}
  f(\gamma P_{(v,l)}) &=& \sum_{(u,k)} \sum_{\tau \in G/H_u}
                        j(\langle \tau P_{(u,k)}, P^t_{(t,j)} \rangle_F) \left( \tau P^*_{(u,k)} \right) \left( \gamma P_{(v,l)} \right) \\
                    &=&   j(\langle \gamma P_{(v,l)}, P^t_{(t,j)} \rangle_F) \\
                   &=&   \left( \lambda^{\NT,j}_{A,F}(P^t_{(t,j)} ) \right) ( \gamma P_{(v,l)}).
\end{eqnarray*}
\end{proof}

\mpar{Proof of rat theorem}
\subsection{The proof of Theorem \ref{rat theorem}}\label{Proof of rat theorem}


We set $M_{A,F}=h^1(A_F)(1)$ and recall that the approach used in \cite{bufl01} in order to formulate the conjecture $C(A, \QG)$ relies upon the theory of
categories of virtual objects. Although focused on the study of $p$-parts of the relevant equivariant Tamagawa numbers for prime numbers $p$, the exact same techniques
involved in the proof of \cite[Prop. 4.2]{bmw} allow one to translate this more technical language into the one of refined Euler characteristics employed throughout
this article. For this reason, we will avoid any explicit mention of categories of virtual objects throughout this proof. Furthermore, since we only need to consider
the relevant realisations and cohomology spaces associated to $M_{A,F}$ as modules over the semisimple algebras $\bq[G]$ or $\br[G]$, we may directly reformulate
$C(A, \QG)$ in terms of the determinants of certain endomorphisms of free $\br[G]$-modules, which is what we do in the sequel.

The leading term $L^*(M_{A,F},0)$ at $s=0$ of the $\bq[G]$-equivariant motivic $L$-function of $M_{A,F}$ is given by
$\sum_{\chi\in{\rm Ir}(G)}e_\chi L^*(A,\check{\chi},1)$.
We let once again
$\lambda^{NT}_{A,F}:\br\otimes_\bz A^t(F)\to \br\otimes_\bz A(F)^*$ denote the canonical isomorphism induced by the N\'{e}ron-Tate height pairing and 
\[\alpha_{A,F}:\br\otimes_\bq\bigoplus_{v\in S_\infty^F}H^0(F_v,H_v(M_{A,F}))\to
\br\otimes_\bq H_{{\rm dR}}(M_{A,F})/F^0\] denote the canonical period isomorphism described by Deligne in \cite{deligne} (see also \cite[Sec. 3]{bufl01}
for a general description of the modules involved and \cite[Sec. 4.3]{bmw} for more details in the relevant special case).

We now note that the $\bq[G]$-modules $X:=\bq\otimes_\bz A^t(F)$ and $Y:=\bq\otimes_\bz A(F)^*$, resp.
$Z:=\bigoplus_{v\in S_\infty^F}H^0(F_v,H_v(M_{A,F}))$ and $W:=H_{{\rm dR}}(M_{A,F})/F^0$,
are isomorphic (as a consequence, for instance, of \cite[p. 110]{cf}) and hence, since $\bq[G]$ is semisimple, there exist $\bq[G]$-modules $M$ and $N$
with the property that both $X\oplus M\cong Y\oplus M$ and $Z\oplus N\cong W\oplus N$ are free $\bq[G]$-modules. In the sequel we (choose bases and so) fix identifications of
$X\oplus M$, $Y\oplus M$, $Z\oplus N$ and $W\oplus N$ with direct sums of copies of $\bq[G]$ and hence regard $\lambda^{NT}_{A,F}\oplus{\rm id}_{\br\otimes_\bq M}$ and
$\alpha_{A,F}\oplus{\rm id}_{\br\otimes_\bq N}$ as elements of $K_1(\br[G])$.

The validity of conjecture $C(A,\QG)$ is then equivalent
to the containment
\[L^*(M_{A,F},0)/({\rm det}_{\br[G]}(\alpha_{A,F}\oplus{\rm id}_{\br\otimes_\bq N}){\rm det}_{\br[G]}(\lambda^{NT}_{A,F}\oplus{\rm id}_{\br\otimes_\bq M}))\in\bq[G]^\times.\]
But it is clear from the proof of Lemma \ref{eval lambda NT} that, independently of our choice of fixed identifications,
\begin{equation}\label{regulatorforrat} {\rm det}_{\br[G]}(\lambda^{NT}_{A,F}\oplus{\rm id}_{\br\otimes_\bq M})/\sum_{\chi\in{\rm Ir}(G)}e_\chi\lambda_\chi(\mathcal{P},\mathcal{P}^t)\in\bq[G]^\times,\end{equation}
and it is straightforward to deduce from the proof of \cite[Lemma 4.5]{bmw} that, independently of our choice of fixed identifications,
\begin{equation}\label{periodforrat} {\rm det}_{\br[G]}(\alpha_{A,F}\oplus{\rm id}_{\br\otimes_\bq N})/\sum_{\chi\in{\rm Ir}(G)}e_\chi\frac{w_\infty(k)^d\cdot \Omega(A/k)}{\tau^*(\Qu,  \ind_k^\Qu(\chi))^d}\in\bq[G]^\times.\end{equation}
The equalities (\ref{regulatorforrat}) and (\ref{periodforrat}), combined with the fact that the Euler factors involved in the truncation of each of the leading terms
$L_{S_r}^*(A, \check{\psi}, 1)$ live by definition in $\bq[G]^\times$, therefore imply
that the validity of $C(A, \QG)$ is equivalent to the containment
\[
\sum_{\psi \in \widehat{G}} \frac{\calL^*_\psi}{\lambda_\psi(\calP, \calP^t)} e_\psi \in \QG^\times.
\]
By \cite[Lem.~2.9]{bleyone} this containment is equivalent to the explicit condition described in Theorem \ref{rat theorem}, as required.

\subsection{The proof of Theorem \ref{max theorem}}\label{Proof of max theorem}

We assume now that C$(A,\bq[G])$ is valid and proceed to prove the explicit interpretation of $C_p(A,\mathcal{M})$ claimed in Theorem \ref{max theorem}.
We begin by noting that, for any fixed isomorphism of fields $j:\bc\to\bc_p$, the respective maps $j_{G,*}$ restrict to give the vertical arrows in a natural
commutative diagram with exact rows of the form
\begin{equation}\label{maxdiagram}
\begin{CD} K_0(\bz[G],\bq[G])_{{\rm tor}} @> >> K_0(\bz[G],\bq[G]) @> \mu >> K_0(\calM,\bq[G])\\
@V j_{G,*} VV  @V j_{G,*} VV @V j_{G,*} VV \\
K_0(\bz_p[G],\bq_p[G])_{{\rm tor}} @>  >> K_0(\bz_p[G],\bq_p[G]) @> \mu_p >> K_0(\calM_p,\bq_p[G]).\\
\end{CD}\end{equation}
We note that the exactness of the rows follows from \cite[Lemma 11]{bufl01}. We now proceed to prove several useful results.

\begin{lem}\label{maxtorsion} C$_p(A,\calM)$ holds if and only if $j_{G,*}\bigl(T\Omega\bigl(h^1(A_{F})(1),\ZZ[G]\bigr)\bigr)$ belongs to $K_0(\bz_p[G],\bq_p[G])_{{\rm tor}}$.
\end{lem}
\begin{proof} The equality $T\Omega\bigl(h^1(A_{F})(1),\calM\bigr)=\mu\bigl(T\Omega\bigl(h^1(A_{F})(1),\ZZ[G]\bigr)\bigr)$ proved in \cite[Th. 4.1]{bufl01}
combines with the commutativity of the right-hand square of diagram (\ref{maxdiagram}) to imply that 
\[j_{G,*}\bigl(T\Omega\bigl(h^1(A_{F})(1),\calM\bigr)\bigr)=\mu_p\bigl(j_{G,*}\bigl(T\Omega\bigl(h^1(A_{F})(1),\ZZ[G]\bigr)\bigr)\bigr).\]
The exactness of the bottom row of diagram (\ref{maxdiagram}) thus completes the proof.
\end{proof}

\begin{lem}\label{maxunits}C$_p(A,\calM)$ holds if and only if
$\sum_{\psi \in \widehat{G}} \frac{\calL_\psi^*}{j(\lambda_\psi(\calP, \calP^t) )\epsilon_\psi(\Phi) \delta_\psi}
e_\psi\in\calM_p^\times$. 
\end{lem}
\begin{proof} Lemma \ref{maxtorsion} combines with equalities (\ref{tomega}) and (\ref{chiGp eq}) to imply that C$_p(A,\calM)$ holds if and only if  
\[\mu_p\left(\delta_{G,p}\left(\sum_{\psi \in \widehat{G}} \frac{\calL_\psi^*}{j(\lambda_\psi(\calP, \calP^t)) \epsilon_\psi(\Phi) \delta_\psi}
e_\psi\right)\right)=0. 
\]

We next note that the respective maps $\delta_{G,p}$ induce vertical (bijective) arrows in a commutative diagram of the form
\[\begin{CD} \QpG^\times / \ZpG^\times @> >> \QpG^\times / \calM_p^\times\\
@V  VV  @V  VV \\
K_0(\bz_p[G],\bq_p[G]) @> \mu_p >> K_0(\calM_p,\bq_p[G])\\.
\end{CD}\] This completes the proof of the Lemma.
\end{proof}

\mpar{unit lemma}
\begin{lem}\label{unit lemma}
  $\varepsilon_\psi(\Phi) \in \Zp[\psi]^\times$.
\end{lem}

\begin{proof}
The map $\Phi \otimes \calM_p \colon M \otimes_{\ZpG} \calM_p \lra A^t(F)_p \otimes_{\ZpG} \calM_p$ is an isomorphism of
$\calM_p$-modules. Since $\calM_p$ contains the $\QpG$-rational idempotents,
\[
\Phi \otimes \Zp[\psi] \colon M \otimes_{\ZpG} \Zp[\psi] \lra A^t(F)_p \otimes_{\ZpG} \Zp[\psi]
\]
is an isomorphism of $\Zp[\psi]$-modules. It is easy to see that $\Phi \otimes \Zp[\psi]$ is represented by $\psi\left( \Phi_{t_\psi}\right)$.
\end{proof}

We now proceed to give the proof of Theorem \ref{max theorem}. 

The equivalence of (i) and (iv) follows directly upon combining Lemmas \ref{maxunits} and \ref{unit lemma}.

Furthermore it is straightforward to compute the valuation of each element $\delta_\psi$. One has $v_{\frp_\psi}(\delta_\psi) = b_\psi$ with
$b_\psi$ defined as in Theorem \ref{max theorem}, and hence (iii) and (iv) are clearly equivalent.

In order to prove the equivalence of (i) and (ii), we will use (a special case of) a general fact which we now describe. If $H$ is any subgroup
of $G$, we write $\rho^G_H \colon K_0(\ZpG, \QpG) \lra  K_0(\Zp[H], \Qp[H])$ for the natural
restriction map and $q^H_0 \colon K_0(\Zp[H], \Qp[H]) \lra  K_0(\Zp, \Qp)$ for the
natural map induced by sending an element $[P,\phi,Q]$ of $K_0(\bz_p[H],\bq_p[H])$ to the element $[P^H,\phi^H,Q^H]$ of $K_0(\bz_p,\bq_p)$.
By \cite[Thm. 4.1]{whitehead} one then has that
\mpar{intersection}
\begin{equation}\label{intersection}
K_0(\bz_p[G],\bq_p[G])_{{\rm tor}}=\bigcap_{H\leq G}{\rm ker}(q_0^H\circ\rho_H^G).
\end{equation}

The functoriality properties of the element $T\Omega\bigl(h^1(A_{F})(1),\ZZ[G]\bigr)$ with respect to the maps $\rho_H^G$ and $q_0^H$ proved in
\cite[Prop.~4.1]{bufl01} then imply that, for any subgroup $H$ of $G$, 
\[ (q_0^H\circ\rho_H^G)\bigl(j_{G,*}\bigl(T\Omega\bigl(h^1(A_{F})(1),\ZZ[G]\bigr)\bigr)\bigr)= j_{0,*}\bigl(T\Omega\bigl(h^1(A_{F^H})(1),\ZZ\bigr)\bigr),\]
and so Lemma \ref{maxtorsion} combines with (\ref{intersection}) to imply that C$_p(A,\calM)$ holds if and only if, for every intermediate field $L$
of $F/k$, the element $j_{G_{L/L},*}\bigl(T\Omega\bigl(h^1(A_{L})(1),\ZZ\bigr)\bigr)$ vanishes, that is, if and only if the $p$-part of the eTNC holds for the pair
$\bigl(h^1(A_{L})(1),\ZZ\bigr)$. Noting that it is easy to check that the set of data $(A/L,L/L,p)$ satisfies all the hypotheses of Theorem \ref{ZpG theorem}
for any such field $L$ (see for instance \cite[Lem. 3.4]{bmwselmer} for a proof of a more general assertion), all that
is left to do in order to prove the equivalence of (i) and (ii) is to apply Theorem \ref{ZpG theorem}. Indeed,
any choice of $\bz$-bases $\{Q_i\}$ and $\{R_j\}$ of $A(L)$ and $A^t(L)$ respectively satisfy condition (\ref{global points}) for the set of data $(A/L,L/L,p)$,
while an explicit computation proves that $\frac{\tau^*(\Qu,  \ind_L^\Qu({\bf{1}}_{G_{L/L}})) }{w_\infty(L)}=\sqrt{|d_L|}$.

\mpar{Proof of Cor 1}
\subsection{The proof of Corollary \ref{Cor 1}}\label{Proof of Cor 1}

For brevity we set
\[
\lambda_\psi := \lambda_\psi(\calP, \calP^t), \quad \epsilon_\psi := \epsilon_\psi(\Phi), \quad u := \sum_{\psi \in \widehat{G}}
\frac{\calL^*_\psi}{\lambda_\psi \epsilon_\psi \delta_\psi} e_\psi.
\]
By Theorem \ref{ZpG theorem} the validity of $C_p(A, \ZG)$ is equivalent to the containment $u \in \ZpG^\times$, which we assume holds throughout the proof.
We also let $\varepsilon \colon \ZpG \lra \Zp$ denote the augmentation map.

We begin by noting that claim (ii) is just the $\psi = \trivchar$ special case of Lemma \ref{unit lemma}, and proceed now to deduce claim (iii) from it.
One clearly has that 
$\calL_\trivchar^* / ( \lambda_\trivchar \epsilon_\trivchar \delta_\trivchar) =\varepsilon(u)\in \Ze_p^\times$ with $\delta_\trivchar$ equal by
definition to 1, while a straightforward computation shows that $\frac{\tau^*(\Qu,  \ind_k^\Qu({\bf{1}}_{G})) }{w_\infty(k)}=(-1)^{|S_r|}\sqrt{|d_k|}$.
Claim (ii) therefore indeed implies that \begin{equation}\label{vtoaugofu}v=\calL_\trivchar^* / \lambda_\trivchar=\varepsilon(u)\cdot\epsilon_\trivchar=\varepsilon(u)\cdot\epsilon\end{equation}
belongs to $\Ze_p^\times$, as required.

In order to prove the remaining claims, we first note that, if ${\rm rk}(A(F^J))=|G/J|{\rm rk}(A(k))$ for every subgroup J of $G$, then $h=0$ by
Proposition \ref{MW iso} while $\Phi$ can be chosen to be the identity matrix by property (P2) and each element $\delta_\psi$ is simply equal to
1 by convention. In any such case, claim (i) therefore reduces to the trivial statement $\calL=u\in\bz_p[G]$ while claim (iv) simply reads
$u \equiv  v (\mod I_{G, p})$ and follows directly from (\ref{vtoaugofu}). We therefore may and will henceforth assume that
the inequality ${\rm rk}(A(F^J))\leq |G/J|{\rm rk}(A(k))$ of Corollary \ref{ranks} is strict for some subgroup $J$ of $G$. We recall that $H=H_{t_0}$ denotes
the smallest non-trivial subgroup of $G$ with the property that $m_H \neq 0$.

In order to prove claim (i), we note first that for each $\psi \in \widehat{G}$ we have
\[
\psi |_H \ne 1 \iff \ker(\psi) \sseq H_{t_0} \text{ and } \ker(\psi) \ne H_{t_0} \iff t_\psi > t_0.
\]
From the definitions of $\epsilon_\psi$ and $\delta_\psi$ we immediately deduce that, for each $\psi \in \widehat{G}$
such that $\psi |_H \ne 1$,
\[
\epsilon_\psi = 1, \quad \delta_\psi = \delta := \prod_{j=0}^{t_0} \left( \sigma^{p^j} - 1 \right)^{m_j}.
\]
Since $\delta e_\psi = 0$ for each $\psi$  such that $\psi |_H = 1$ we deduce that $\calL = \delta u \in \delta \ZpG\subseteq I_{G,p}^h$, as required.

Finally, claim (iv) follows from (\ref{vtoaugofu}) because $u$ is clearly congruent to $\varepsilon(u)=v/\epsilon$ modulo $I_{G,p}$ and therefore 
$\calL = \delta u$ is congruent to $\delta\frac{v}{\epsilon}$ modulo $I_{G,p}^{h+1}$, as required.

\mpar{Proof of Ann theorem}
\subsection{The proof of Theorem \ref{Ann theorem}}\label{Proof of ann theorem}

We begin by defining a (free) $\bz_p[G]$-submodule
\[
P := \bigoplus_{j \in [m_n]} \ZpG P_{(n,j)}^*
\]
of $A(F)_p^*$ and then fix, as we may, an injective lift $\kappa \colon P \lra \Sel_p(A_F)^\vee$ of the inclusion
$P \sseq A(F)_p^*$ through the canonical projection of (\ref{Sel ses}).
We also fix, as we may, a representative of the perfect complex $\Cf$ of the form $C^1\to C^2$ in which both $C^1$ and $C^2$ are finitely generated,
cohomologically-trivial $\bz_p[G]$-modules.
We then obtain a commutative diagram with exact
rows and columns of the form \begin{equation}\minCDarrowwidth0.9em\begin{CD}\label{quotientingsha}
  @. 0 @. 0 @. 0 @. 0 @.  \\
 @. @V VV @V VV @V VV @V VV @. \\
0 @> >> \bigoplus_{j} \Zp[G] P_{(n,j)}^t @=
\bigoplus_{j} \Zp[G] P_{(n,j)}^t @> 0 >>
{\rm im}(\kappa) @= {\rm im}(\kappa) @> >> 0\\
 @. @V VV @V VV @V  VV @V VV @. \\
0 @> >> A^t(F)_p @> >> C^1 @>
 >>
C^2 @> >> {\rm Sel}_p(A_F)^\vee @> >> 0\\
 @. @V VV @V VV @V VV @V VV @. \\
0 @> >> N @> >>
D^1 @> >>
D^2 @> >> {\rm cok}(\kappa) @>  >> 0\\
 @. @V VV @V VV @V VV @V VV @. \\
  @. 0 @. 0 @. 0 @. 0 @.
\end{CD}\end{equation}
in which we have set
\[
N:=\bigoplus_{t<n,j\in[m_t]} \Zp[G/H_t]
\]
and we simply define the arrow ${\rm im}(\kappa)\to C^2$, as we may since ${\rm im}(\kappa)$ is a free $\bz_p[G]$-module, by the commutativity of the upper right square.

The $\ZpG$-modules $D^1$ and $D^2$ are finitely generated and cohomologically-trivial,
and hence the central arrow of the bottom row of this diagram defines an object $D^\bullet$ of $D^p(\bz_p[G])$ which is acyclic outside of
degrees 1 and 2 and has identifications of $H^1(D^\bullet)$ with $N$ and of $H^2(D^\bullet)$ with
${\rm cok}(\kappa)$. We analogously define an object $B^\bullet$ of $D^p(\bz_p[G])$ represented by 
the perfect complex of $\bz_p[G]$-modules
\[\bigoplus_{j} \ZpG P_{(n,j)}^t\stackrel{0}{\lra}{\rm im}(\kappa).\]

Following \cite[Sec.~2.1.4]{omac} we next define an idempotent $e_N:=\sum_{\psi\in\Upsilon_N}e_\psi$ in $\bq_p[G]$ by letting $\Upsilon_N$ be the subset
of $\widehat{G}$ comprising characters $\psi$ with the property that $e_\psi(\Cp \tensorZp N)=0$. For any object
$C^\bullet$ of $D^p(\bz_p[G])$ we then obtain an object $e_N C^\bullet:=e_N\bz_p[G]\otimes_{\bz_p[G]}^{\mathbb{L}}C^\bullet$ of $D^p(e_N\bz_p[G])$. In particular, the exact triangle represented by diagram
(\ref{quotientingsha}) induces an exact triangle in $D^p(e_N\bz_p[G])$ of the form
\begin{equation}\label{Ntriangle}e_NB^\bullet\lra e_N\Cf\lra e_ND^\bullet\lra e_NB^\bullet[1].\end{equation} But $e_ND^\bullet\otimes_{e_N\bz_p[G]}e_N\bc_p[G]$ is
acyclic and an immediate application of the additivity criterion of \cite[Cor.~6.6]{additivity} to triangle (\ref{Ntriangle}) implies that one has
\begin{align}\label{Nadditivity}-\chi_{e_N\bz_p[G],e_N\bc_p[G]}(e_ND^\bullet,0)=&-\chi_{e_N\bz_p[G],e_N\bc_p[G]}(e_N\Cf,e_N(\lambda^{\NT,j}_{A,F})^{-1})\notag\\
&+\chi_{e_N\bz_p[G],e_N\bc_p[G]}(e_N B^\bullet,\lambda')\end{align} where $\lambda'$ denotes the canonical isomorphism \begin{align*}e_N(\bc_p\otimes_{\bz_p}{\rm im}(\kappa))
=&e_N(\bc_p\otimes_{\bz_p}\Sel_p(A_F)^\vee)\\ \stackrel{e_N(\lambda^{\NT,j}_{A,F})^{-1}}{\lra}&e_N(\bc_p\otimes_{\bz_p}A^t(F)_p)=e_N(\bc_p\otimes_{\bz_p}\bigoplus_{j} \ZpG P_{(n,j)}^t).\end{align*}
If we now write $\varphi:\bigoplus_{j} \ZpG P_{(n,j)}^t\to\bigoplus_{j} \ZpG P_{(n,j)}^*$ for the canonical isomorphism that maps an element $P_{(n,j)}^t$ to the element
$P_{(n,j)}^*$, then one finds that \begin{align}\label{Nregulator}\chi_{e_N\bz_p[G],e_N\bc_p[G]}(e_N B^\bullet,\lambda')=&\delta_{e_N\bz_p[G],e_N\bc_p[G]}
({\rm det}_{e_N\bc_p[G]}(\lambda'\circ e_N(\bc_p\otimes_{\bz_p}(\kappa\circ\varphi))))\notag\\ =&-\delta_{e_N\bz_p[G],e_N\bc_p[G]}(\sum_{\psi \in \Upsilon_N}  j(\lambda_\psi(\calP, \calP^t))   e_\psi),\end{align}
where the last equality follows from Lemma \ref{eval lambda NT}.

The assumed validity of $C_p(A, \ZG)$ therefore combines via (\ref{reform eq}) with equalities (\ref{Nadditivity}) and (\ref{Nregulator}) to imply that, in the
terminology of \cite[\S 2.3.2]{omac}, the element \[j_*(\sum_{\psi \in \Upsilon_N}   \calL^*_\psi    \lambda_\psi(\calP, \calP^t)^{-1}   e_\psi)=j_*(\calL)\] of
$e_N\bc_p[G]$ is a characteristic element for $e_ND^\bullet$. The result \cite[Lem. 2.6]{omac} therefore implies that there exists a characteristic element $\calL'$ for
$D^\bullet$ in $\bc_p[G]$ with the property that $e_N\calL'=j_*(\calL)$. Since $D^\bullet$ is clearly an admissible complex
of $\bz_p[G]$-modules (in the terminology of \cite[\S 2.1.1]{omac}), the results of \cite[Cor.~3.3]{omac} therefore imply
that the element $j_*(\calL)$ belongs to the ideal $I_{G,p}^{\tilde h}$ of $\bz_p[G]$, with $\tilde h := \dim_\Qp (\Qp \tensorZp \cok(\kappa)_G)$, and furthermore
generates ${\rm Fitt}_{\ZpG}(\cok(\kappa))$. To proceed with the proof, we first note that
\[\tilde h = \dim_\Qp (\Qp \tensorZp \cok(\kappa)_G)=\dim_\Qp (\Qp \tensorZp \bigoplus_{t<n}\bz_p[G/H_t]_G^{<m_t>})=\sum_{t<n}m_t=h.\]
We have hence proved that $j_*(\calL)$ belongs to $I_{G,p}^h$, as stated in claim (i) of Theorem \ref{Ann theorem}. Furthermore, $\Pi:={\rm im}(\kappa)$ is clearly a
finitely generated, free $\bz_p[G]$-submodule of ${\rm Sel}_p(A_F)^\vee$ of maximal rank $m_n$, and so the fact that the element $j_*(\mathcal{L})$
generates ${\rm Fitt}_{\bz_p[G]}({\rm cok}(\kappa))$ proves claim (iii) of Theorem \ref{Ann theorem}. To complete the proof, it is enough to note that,
since ${\rm im}(\kappa)$ is torsion-free,
the canonical composite homomorphism
\[
\Sha_p(A_F)^\vee\stackrel{\sim}\to({\rm Sel}_p(A_F)^\vee)_{{\rm tor}}\subseteq{\rm Sel}_p(A_F)^\vee\to{\rm cok}(\kappa)
\]
is injective and hence that one has that
\[
{\rm Fitt}_{\bz_p[G]}({\rm cok}(\kappa))\subseteq{\rm Ann}_{\bz_p[G]}(\Sha_p(A_F)^\vee).
\]
Recalling finally that the Cassels-Tate pairing induces a canonical isomorphism between
$\Sha_p(A_F)^\vee$ and $\Sha_p(A_F^t)$ completes the proof of claim (ii) and thus of Theorem \ref{Ann theorem}.

\section{Examples}\label{examples}

In this section we gather some evidence, mostly numerical, in support of conjecture C$_p(A, \ZG)$.
Our aim is to verify statements that would not follow in an straightforward manner from the validity of the Birch and Swinnerton-Dyer conjecture for all intermediate fields
of $F/k$. Because of the equivalence of statements (i) and (ii) in Theorem \ref{max theorem} we therefore choose not to focus on presenting evidence for conjecture
C$_p(A, \calM)$ (although we also used our MAGMA programs to produce numerical evidence
for C$_p(A, \calM)$ by verifying statement (iii) of Theorem \ref{max theorem}). 

\smallskip
Throughout this section $A$ will always denote an elliptic curve.

\subsection{Verifications of conjecture C$_p(A, \ZG)$}

For the verification of C$_p(A, \ZG)$ using Theorem \ref{ZpG theorem} it is necessary to have explicit knowledge
of a map $\Phi$ that represents the extension class $\delta_{A,F,p}$. Whenever $A(F)_p$ is not projective as a $\ZpG$-module
we are currently not able to numerically compute $\Phi$, so we only deal with examples in which $A(F)_p$ is
projective. To the best of our knowledge there are currently three instances of theoretical evidence (in situations in which our fixed cyclic extension $F/k$ is not trivial):

\begin{itemize}
\item In \cite{bleythree}, it is shown that for each elliptic curve $A/\Qu$ with $L(A/\Qu, 1) \ne 0$ there are infinitely many
primes $p$ and for each such prime $p$ infinitely many (cyclic) $p$-extensions $F/\Qu$ such that C$_p(A, \Ze[\Gal(F/\bq)])$ holds.
All of these examples satisfy our hypotheses and are such that $A(F)_p$ vanishes.

\item In \cite[Th.~1.1]{bmw}, C$_p(A, \Ze[\Gal(F/\bq)])$ is proved for certain elliptic curves $A/\Qu$,
where $F$ denotes the Hilbert $p$-classfield of an imaginary quadratic field $k$. This result combines with the functoriality properties of the eTNC to
imply the validity of C$_p(A, \Ze[\Gal(F/k)])$. In these examples one has that $A(F)_p$ is a free $\Zp[\Gal(F/k)]$-module of rank one.

\item In \cite[Cor.~6.2]{bmw}, certain $S_3$-extensions $F/K$ are considered. Let $k$ and $L$ denote the quadratic and cubic
subfield of $F/K$ respectively. Under certain additional assumptions it is then shown that the validity of the Birch and Swinnerton-Dyer conjecture for $A$ over the fields $k, K$ and $L$ implies
the validity of C$_p(A, \Ze[\Gal(F/K)])$. Again by functoriality arguments, the validity of
C$_p(A, \Ze[\Gal(F/k)])$ follows. We note that the assumptions are such that one again has that $A(F)_p$ is a free $\Zp[\Gal(F/k)]$-module of rank one.
\end{itemize}

In the rest of this section we are concerned with numerical evidence. In \cite[Sec.~6]{bleyone} there is a list of examples
of elliptic curves $A/\Qu$ and dihedral extensions $F/\Qu$ of order $2p$ for which C$_p(A, \Ze[\Gal(F/\Qu)])$ is numerically verified.
Here the quadratic subfield $k$ is real and $A(F)_p$ vanishes. Again by functoriality arguments we obtain examples where C$_p(A, \Zp[\Gal(F/k)])$ is numerically
verified. There are two further analogous numerical verifications in dihedral examples in \cite[Sec.~6.3]{bmw}, one of degree $10$ and one of 
degree $14$, both of them with the property that $A(F)_p$ is a free $\Zp[\Gal(F/k)]$-module of rank one.

In the following we fix an odd prime $p$ and let $q$ denote a prime such that $q \equiv 1 (\mod p)$. We let $F$ denote
the unique subfield of $\Qu(\zeta_q)/\Qu$ of degree $p$ and take $k$ to be $\Qu$. 
For $p \in \{3,5,7\}$ and $q < 50$ we went through the list of semistable elliptic curves of rank one and conductor $N < 200$ and
checked numerically whether $L(A/\Qu, \chi, 1)=0$ and  $L'(A/\Qu, \chi, 1) \ne 0$ for a non-trivial character $\chi$ of $G$, and in addition,
whether our hypotheses are satisfied.
This resulted in a list of $50$ examples ($27$ for $p=3$, $20$ for $p=5$ and $3$ for $p=7$). In each of these examples we could find
a point $R$ such that $A(F)_p = \ZpG R$ and numerically verify conjecture $C_p(A, \ZG)$. 

We now describe in detail an example with $[F:\Qu] = 7$.
Let $A$ be the elliptic curve
\[
A \colon y^2 + xy + y = x^3 + x^2 - 2 x.
\]
This is the curve 79a1 in Cremona's notation. It is known that $A(\Qu)$ is free of rank one generated by
$P_1 = (0,0)$ and that $\Sha(A_\Qu) = 0$. Moreover it satisfies the hypotheses used throughout the paper.

We take $p = 7$ and let $F$ be the unique subfield of $\Qu(\zeta_{29})$ of degree $7$. Explicitly,
$F$ is the splitting field of
\[
f(x) =  x^7 + x^6 - 12 x^5 - 7 x^4 + 28 x^3 + 14 x^2 - 9 x + 1
\]
and we let $\alpha$ denote a root of $f$.
Using the MAGMA command {\it Points} it is easy to find a point $R$ of infinite order in $A(F) \setminus A(\Qu)$,
\begin{eqnarray*}
R =  && \left( \frac{1}{17} (31\alpha^6 + 23\alpha^5 - 373\alpha^4 - 135\alpha^3 + 814\alpha^2 + 372\alpha - 86),\right.\\
     && \left. \frac{1}{17} (-35\alpha^6 - 83\alpha^5 + 380\alpha^4 + 771\alpha^3 - 811\alpha^2 - 1321\alpha + 232)  \right).
\end{eqnarray*}
By Proposition \ref{MW iso} we know that $A(F)_p$ is a permutation module, hence $A(F)_p \simeq \Zp[G]$. Furthermore,
\cite[Prop. 3.1]{bmwselmer} now implies that $\Sha_p(A_F) = 0$. 

We set $Q_1 := \Tr_{F/\Qu}(R) = ( \frac{3}{4}, -\frac{3}{8}  )$ and easily verify that $Q_1 = -4 P_1$.
We checked numerically that $\ZpG R = A(F)_p$.
\mpar{CHECK !!}

Computing numerical approximations to the leading terms using  Dokchitser's MAGMA implementation of 
\cite{Dok} we obtain the following vector for  $\left( \calL^*_\chi / \lambda_\chi(\calP, \calP^t) \right)_{\chi \in \hat{G}}$
\begin{eqnarray*}
&&(-0.077586206896551724152  ,\\
&& -0.49999999999999999992 + 2.1906431337674115362 i,\\
&& -0.49999999999999999996 + 0.62698016883135191886 i,\\
&& -0.49999999999999999998 - 0.24078730940376432202 i, \\
&& -0.49999999999999999992 - 2.1906431337674115362 i, \\
&& -0.49999999999999999996 - 0.62698016883135191886 i,\\
&& -0.49999999999999999998 + 0.24078730940376432202 i )
\end{eqnarray*}
This is very close to
\begin{eqnarray*}
&&(-9/116,
\zeta_7^3 + \zeta_7^2 + \zeta_7,
-\zeta_7^5 - \zeta_7^4 - \zeta_7 - 1,
-\zeta_7^5 - \zeta_7^3 - \zeta_7 - 1,\\
&&-\zeta_7^3 - \zeta_7^2 - \zeta_7 - 1,
\zeta_7^5 + \zeta_7^4 + \zeta_7,
\zeta_7^5 + \zeta_7^3 + \zeta_7)
\end{eqnarray*}
It is now easy to verify the rationality conjecture C$(A, \QG)$ by the criterion of Theorem \ref{rat theorem}.
Moreover, the valuations of $-9/116$ and $\zeta_7^3 + \zeta_7^2 + \zeta_7$ at $\frp_\chi$
are $0$, so that by Theorem \ref{max theorem} we deduce the validity of C$_p(A, \calM)$.
Finally, one easily checks that $-9/116 \equiv \zeta_7^3 + \zeta_7^2 + \zeta_7 (\mod (1-\zeta_7))$, so that 
the element in (\ref{ZpG criterion}) is actually a unit in $\ZpG$, thus (numerically) proving C$_p(A, \ZG)$.

\bigskip

\subsection{Evidence in support of conjecture C$_p(A, \ZG)$}

In this subsection we collect evidence for statements that we have shown to follow from the validity of C$_p(A, \ZG)$ and focus on situations in which $A(F)_p$ is not 
$\ZpG$-projective. In particular, we aim to verify claim (i) of Theorem \ref{Ann theorem}. 
Since we can neither compute the module $\Sha_p(A_F)$ nor a map $\Phi$ as required, we are not able to verify any other claim of either Corollary \ref{Cor 1} or Theorem \ref{Ann theorem}.

Again we want to focus on evidence which goes beyond implications of the Birch and Swinnerton-Dyer conjecture for $A$ over all intermediate fields of $F/k$.
We assume the notation of  Theorem \ref{Ann theorem}, so in particular set $h = \sum_{t < n} m_t$. If $k=\Qu$ and $m_n = 0$, then 
the element $\calL$ is essentially the Mazur-Tate modular element (see \cite{mazurtate}) and the validity of the Birch and Swinnerton-Dyer conjecture would imply that
it belongs to $I_{G,p}$ in the type of situations under consideration. Hence, if $m_n = 0$, we only searched for examples where $h > 1$.

If $F/k$ is cyclic of order $p$, then $I_{G,p}^h = (\sigma  - 1)^h \calM_p$ for all $h \ge 1$. Letting $u$ and $\delta$ denote
the elements defined in the proof of Corollary \ref{Cor 1} we hence note that, if C$_p(A, \calM)$ is valid, then $u \in \calM^\times$ and the proof
of Corollary \ref{Cor 1} clearly shows that $\calL = \delta u$ is contained in $\delta \calM_p = I_{G, p}$.
We therefore further restricted our search for interesting examples to cases where $[F : k] = p^n$ with $n \ge 2$.

Restricted by the complexity of the computations and the above considerations 
we are therefore lead to consider the following types of examples:
\begin{itemize}
\item[(i)] $A(F)_p \simeq  \ZZ_p^{m_0} \oplus \Zp[G/H_1]^{m_1} \oplus \ZpG^{m_2}, \quad [F : \Qu] = 3^2,$ \\
$(m_0, m_1, m_2) = (1,1,0)$,
\item[(ii)] $A(F)_p \simeq  \ZZ_p^{m_0} \oplus \Zp[G/H_1]^{m_1} \oplus \ZpG^{m_2}, \quad [F : \Qu] = 3^2,$ 
\\$(m_0, m_1, m_2) = (m_0,0,0), \quad m_0 \ge 2$.
\end{itemize}

We note that, whenever $\Sha_p(A_F)$ is trivial, the validity of C$_p(A, \ZG)$ 
implies via Corollary \ref{Cor 1} that $h$ is the exact order of vanishing, i.e., that $\calL \in I_{G,p}^h \setminus  I_{G,p}^{h+1}$.
However, this need not be true if $\Sha_p(A_F)$ is non-trivial. In such cases, by Theorem \ref{Ann theorem} (iii), $C_p(A, \ZG)$ does predict
that $\calL$ generates the Fitting ideal of $\Sel_p(A_F)^\vee$ since $m_n = 0$ immediately implies $\Pi = 0$. 

Let $q$ denote a prime such that $q \equiv 1 (\mod 3^2)$. We let $F$ denote
the unique subfield of $\Qu(\zeta_q)/\Qu$ of degree $9$ and take $k$ to be $\Qu$. 

We checked two examples of type (i), namely those given by the pairs $(A, q) \in \{ (681c1, 19), (1070a1, 19) \}$. In both cases we were able
to find a points $P_0$ and $P_1$  such that $A(F)_p = \Zp P_0 \oplus \Zp[G/H] P_1$, where $H$ denotes the subgroup of order $3$. Each time we numerically
found that $\Sha_p(A_F) = 0$ (predicted by the Birch and Swinnerton-Dyer conjecture for $A$ over $F$) and verified that $h$ is the precise order of vanishing, as predicted
by Corollary \ref{Cor 1}.

Concerning examples of type (ii) went through the list of semistable elliptic curves of rank $2$ and conductor $N < 750$ and
produced by numerically checking $L$-values and derivatives a list of 12 examples satisfying the necessary hypotheses. In each of these examples we had $h=m_0=2$ and
could numerically verify the containment $\calL \in I_{G,p}^2$. Whenever $\Sha_p(A_F)$ was trivial we also checked that $\calL \not\in I_{G,p}^3$.

We finally present one example in detail.
Let $A$ be the elliptic curve
\[
A \colon y^2 + y = x^3 + x^2 - 2x.
\]
This is the curve 389a1 in Cremona's notation. It is known that $A(\Qu)$ is free of rank two generated by
$P_1 =  (0,0)$ and $P_2 = (-1, 1)$ and that $\Sha(A_\Qu) = 0$. Moreover it satisfies the hypotheses required to apply Theorem \ref{Ann theorem} (see Remark
\ref{shaorpermutation}).

Computing numerical approximations to the leading terms we find that the order of vanishing at each non-trivial character is $0$. 
The rank part of the Birch and Swinnerton-Dyer conjecture therefore predicts that $\rk(A(F)) = 2$.
We checked that $\langle P_1, P_2 \rangle_\Ze$ is $3$-saturated in $A(F)$ and therefore (conjecturally) conclude that
$A(F)_p = \langle P_1, P_2 \rangle_\Zp \simeq \Ze_p^2$.

The Birch and Swinnerton-Dyer conjecture predicts that $|\Sha_p(A_F)| = 81$. We therefore cannot test for the precise
order of vanishing. 

Computing leading terms, periods and regulators we find the following numerical approximations to  
$\left( \calL^*_\chi / \lambda_\chi(\calP, \calP^t) \right)_{\chi \in \hat{G}}$
\begin{eqnarray*}
&& (-1.243243 ,
    1.500000 + 2.598076 i,
    1.500000 - 2.598076 i,\\
&&  0.358440 + 2.032818 i,
    0.286988 - 0.104455 i,
   -3.645429 + 3.058878 i,\\
&&  0.358440 - 2.032818 i, 
    0.286988 + 0.104455 i, 
   -3.645429 - 3.058878 i). 
\end{eqnarray*}
The actual computation was done with a precision of $30$ decimal digits.

This is very close to
\begin{eqnarray*}
&& (-46/37, \quad
    3\zeta_3 + 3, \quad
    -3\zeta_3, \\ 
&&  2\zeta_9^3 - \zeta_9^2 + 2\zeta_9, \quad
    -\zeta_9^4 - 2\zeta_9^3 + 2\zeta_9^2 - 2, \quad
   \zeta_9^5 + 2\zeta_9^4 + 2\zeta_9^3 + \zeta_9^2, \\
&&  -2\zeta_9^5 + \zeta_9^4 - 2\zeta_9^3 - 2\zeta_9^2 + \zeta_9 - 2, \quad
   -\zeta_9^5 - 2\zeta_9^4 + 2\zeta_9^3 - 2\zeta_9,\quad
   2\zeta_9^5 - 2\zeta_9^3 - \zeta_9 - 2 ).
\end{eqnarray*}

It is now easy to verify the rationality conjecture C$(A, \QG)$ by the criterion of Theorem \ref{rat theorem}.
We finally find that 
\[
\calL = -\sigma + 2\sigma2 - \sigma^3 + 2\sigma^5 -2\sigma^6  -2\sigma^7 + 2\sigma^8
\]
and easily check that $\calL \in I_{G, p}^2$.



\begin{thebibliography}{CCFT91}

\bibitem{cf} M. F. Atiyah, C. T. C. Wall,
\newblock Cohomology of Groups, In: `Algebraic Number Theory', J.
W. S. Cassels, A. Fr\"ohlich (eds.), 94-115
\newblock Academic Press, London, 1967.

\bibitem{bleyone} W. Bley,
\newblock Numerical evidence for the equivariant Birch and
Swinnerton-Dyer conjecture,
\newblock Exp. Math. \textbf{20} (2011), 426-456.

\bibitem{bleytwo} W. Bley,
\newblock Numerical evidence for the equivariant Birch and
Swinnerton-Dyer conjecture (part II),
\newblock Math. Comp. \textbf{81} (2012), 1681-1705.

\bibitem{bleythree} W. Bley,
\newblock The equivariant Tamagawa number conjecture and modular symbols,
\newblock Math. Ann. \textbf{356} (2013), 179-190.

\bibitem{additivity} M. Breuning, D. Burns,
\newblock Additivity of Euler characteristics in relative algebraic K-groups,
\newblock Homology, Homotopy App. \ \textbf{7} (2005), 11-36.




\bibitem{whitehead} D. Burns,
\newblock Equivariant Whitehead torsion and refined Euler characteristics,
\newblock CRM Proceedings and Lecture Notes \textbf{36} (2004) 35-59.

\bibitem{leading} D. Burns,
\newblock Leading terms and values of equivariant motivic
L-functions,
\newblock Pure App. Math. Q. {\bf 6} (2010) 83-172 (John Tate Special Issue, Part II).

\bibitem{bufl96} D. Burns, M. Flach,
\newblock Motivic L-functions and Galois module structures,
\newblock Math. Ann. \textbf{305} (1996) 65-102.

\bibitem{bufl01} D. Burns, M. Flach,
\newblock Tamagawa numbers for motives with (non-commutative)
coefficients,
\newblock Documenta Math. \textbf{6} (2001) 501-570.


\bibitem{omac} D. Burns, D. Macias Castillo,
\newblock Organising matrices for arithmetic complexes,
\newblock Int. Math. Res. Notices {\bf 2014}, 10 (2014) 2814-2883.

\bibitem{bmwselmer} D. Burns, D. Macias Castillo, C. Wuthrich,
\newblock A note on Selmer structures,
\newblock submitted for publication.

\bibitem{bmw} D. Burns, D. Macias Castillo, C. Wuthrich,
\newblock On Mordell-Weil groups and congruences between derivatives of Hasse-Weil $L$-functions,
\newblock submitted for publication.


\bibitem{curtisr} C. W. Curtis, I. Reiner,
\newblock Methods of Representation Theory, Vol. I and II,
\newblock John Wiley and Sons, New York, 1987.

\bibitem{deligne} P. Deligne,
\newblock Valeurs de fonctions $L$ et p\'eriodes d'int\'egrales,
\newblock Proc. Sym. Pure Math. {\bf 33}(2) (1979) 313-346.

\bibitem{Dok} T. Dokchitser,
\newblock Computing special values of motivic $L$-functions,
\newblock Experiment. Math.  {\bf 13}(2) (2004) 137-149.

\bibitem{fka} J. Fearnley, H. Kisilevsky,
\newblock Critical values of derivatives of twisted
elliptic L-functions,
\newblock Exp. Math. {\bf 19} (2010) 149-160.

\bibitem{fk} J. Fearnley, H. Kisilevsky,
\newblock Critical values of higher derivatives of twisted
elliptic L-functions,
\newblock Exp. Math. {\bf 21} (2012) 213-222.





\bibitem{G-BSD} B. H. Gross,
\newblock On the conjecture of Birch and Swinnerton-Dyer for elliptic curves with complex multiplication,
\newblock in: Number theory related to Fermat's last theorem
(Cambridge, Mass., 1981), Prog. Math. {\bf 26}, Birkhauser Boston, Mass., 1982, pp. 219-236.






\bibitem{martinet} J. Martinet,
\newblock Character theory and Artin $L$-functions,
\newblock in: Algebraic number fields (ed. A. Fr\"ohlich), pp. 1-87, Academic Press, London, 1977.

\bibitem{mazurtate} B. Mazur, J. Tate,
\newblock Refined Conjectures of the Birch and Swinnerton-Dyer
type,
\newblock Duke Math. J. {\bf 54} (1987) 711-750.




\bibitem{swan} R. G. Swan,
\newblock Algebraic $K$-Theory,
\newblock Springer, New York (1978).





\bibitem{yakovlev} A. V. Yakovlev,
\newblock Homological definability of $p$-adic representations of a ring with power basis,
\newblock Izvestia A N SSSR, ser. Math. {\bf 34} (1970), 321-342. (Russian)

\end{thebibliography}

\noindent
Werner Bley,
Mathematisches Institut der Universit\"at M\"unchen,
Theresienstr.~39,
80333 M\"unchen,
Germany,
E-mail: bley@math.lmu.de

\smallskip

\noindent
Daniel Macias Castillo,
Instituto de Ciencias Matem\'aticas (ICMAT),
Madrid 28049,
Spain,
E-mail: daniel.macias@icmat.es

\end{document}